\documentclass[12pt]{amsart}
\usepackage{amssymb}
\usepackage[shortlabels]{enumitem}
\usepackage[all]{xy}
\usepackage{xcolor}
\usepackage{mathrsfs}
\usepackage[margin=1in]{geometry}
\usepackage{graphicx}
\usepackage[bookmarks, bookmarksdepth=2, colorlinks=true, linkcolor=blue, citecolor=blue, urlcolor=blue]{hyperref}
\usepackage{eucal}
\usepackage{stmaryrd}
\usepackage{contour}
\usepackage[normalem]{ulem}
\usepackage{dsfont}

\makeatletter
\@addtoreset{equation}{section}
\makeatother

\numberwithin{equation}{section}
\newtheorem{theorem}[equation]{Theorem}

\newtheorem{proposition}[equation]{Proposition}
\newtheorem{lemma}[equation]{Lemma}
\newtheorem{corollary}[equation]{Corollary}

\theoremstyle{definition}
\newtheorem{rmk}[equation]{Remark}
\newenvironment{remark}[1][]{\begin{rmk}[#1] \pushQED{\qed}}{\popQED \end{rmk}}
\newtheorem{eg}[equation]{Example}
\newenvironment{example}[1][]{\begin{eg}[#1] \pushQED{\qed}}{\popQED \end{eg}}
\newtheorem{defnaux}[equation]{Definition}
\newenvironment{definition}[1][]{\begin{defnaux}[#1]\pushQED{\qed}}{\popQED \end{defnaux}}
\newtheorem{constaux}[equation]{Construction}
\newenvironment{construction}[1][]{\begin{constaux}[#1]\pushQED{\qed}}{\popQED \end{constaux}}

\newcommand{\bmu}{\mbox{$\raisebox{-0.59ex}
  {$l$}\hspace{-0.18em}\mu\hspace{-0.88em}\raisebox{-0.98ex}{\scalebox{2}
  {$\color{white}.$}}\hspace{-0.416em}\raisebox{+0.88ex}
  {$\color{white}.$}\hspace{0.46em}$}{}}

\newcommand{\cA}{\mathcal{A}}
\newcommand{\fA}{\mathfrak{A}}
\newcommand{\rA}{\mathrm{A}}

\newcommand{\fB}{\mathfrak{B}}
\newcommand{\rB}{\mathrm{B}}

\newcommand{\cC}{\mathcal{C}}

\newcommand{\cE}{\mathcal{E}}

\newcommand{\bF}{\mathbf{F}}
\newcommand{\cF}{\mathcal{F}}
\newcommand{\fF}{\mathfrak{F}}

\newcommand{\fG}{\mathfrak{G}}

\newcommand{\rK}{\mathrm{K}}

\newcommand{\fM}{\mathfrak{M}}

\newcommand{\bN}{\mathbf{N}}

\newcommand{\rN}{\mathrm{N}}

\newcommand{\bQ}{\mathbf{Q}}

\newcommand{\bR}{\mathbf{R}}
\newcommand{\cR}{\mathcal{R}}
\newcommand{\fR}{\mathfrak{R}}

\newcommand{\bS}{\mathbf{S}}

\newcommand{\fS}{\mathfrak{S}}

\newcommand{\fU}{\mathfrak{U}}

\newcommand{\fX}{\mathfrak{X}}

\newcommand{\fY}{\mathfrak{Y}}

\newcommand{\bZ}{\mathbf{Z}}

\newcommand{\fZ}{\mathfrak{Z}}

\newcommand{\fc}{\mathfrak{c}}

\newcommand{\arxiv}[1]{\href{http://arxiv.org/abs/#1}{{\tiny\tt arXiv:#1}}}

\newcommand{\DOI}[1]{\href{http://doi.org/#1}{\color{purple}{\tiny\tt DOI:#1}}}

\contourlength{1pt}

\contourlength{0.8pt}
\newcommand{\myuline}[1]{%
  \uline{\phantom{#1}}%
  \llap{\contour{white}{#1}}%
}
\DeclareMathOperator{\uRep}{\text{\myuline{\rm Rep}}}
\DeclareMathOperator{\uPerm}{\ul{Perm}}

\let\ul\underline

\renewcommand{\phi}{\varphi}

\DeclareMathOperator{\im}{im}
\DeclareMathOperator{\End}{End}
\DeclareMathOperator{\Aut}{Aut}
\DeclareMathOperator{\Hom}{Hom}
\DeclareMathOperator{\Surj}{Surj}
\DeclareMathOperator{\Sh}{Sh}

\DeclareMathOperator{\Sub}{Sub}
\DeclareMathOperator{\Res}{Res}
\DeclareMathOperator{\Ind}{Ind}
\DeclareMathOperator{\Pro}{Pro}
\newcommand{\id}{\mathrm{id}}
\newcommand{\orb}{\mathrm{orb}}
\newcommand{\gen}{\mathrm{gen}}
\newcommand{\Mat}{\mathrm{Mat}}
\newcommand{\op}{\mathrm{op}}
\renewcommand{\Vec}{\mathrm{Vec}}
\newcommand{\Set}{\mathbf{Set}}

\newcommand{\defn}[1]{\emph{#1}}
\newcommand{\bzero}{\mathbf{0}}
\newcommand{\bone}{\mathbf{1}}

\newcommand{\TT}{\mathbb{T}}

\title[Regular categories, oligomorphic monoids, and tensor categories]{Regular categories, oligomorphic monoids, \\ and tensor categories}

\author{Andrew Snowden}
\thanks{The author was supported by NSF grant DMS-2301871.}

\date{March 24, 2024}

\begin{document}

\begin{abstract}
Knop constructed a tensor category associated to a finitely-powered regular category equipped with a degree function. In recent work with Harman, we constructed a tensor category associated to an oligomorphic group equipped with a measure. In this paper, we explain how Knop's approach fits into our theory. The first, and most important, step describes finitely-powered regular categories in terms of oligomorphic monoids; this may be of independent interest. We go on to examine some aspects of this construction when the regular category one starts with is the category of $G$-sets for an oligomorphic group $G$, which yields some interesting examples. 
\end{abstract}

\maketitle
\tableofcontents

\section{Introduction}

\subsection{Overview}

Pre-Tannakian categories are a natural and important class of tensor categories generalizing representation categories of algebraic (super)groups. Deligne \cite{Deligne1} showed that the pre-Tannakian categories of moderate (i.e., exponential) growth are exactly the classical representation categories of algebraic supergroups (in characteristic~0). Constructing pre-Tannakian categories of superexponential growth is a challenging problem. Deligne \cite{Deligne2} (following work of Deligne--Milne \cite{DeligneMilne}) gave an important example of such a category, namely, the interpolation category $\uRep(\fS_t)$ where $t$ is a complex number.

Shortly after Deligne introduced the category $\uRep(\fS_t)$, Knop \cite{Knop1,Knop2} gave a more general construction of pre-Tannakian categories. Recall that a category is \defn{regular} if, roughly speaking, morphisms factor into injective and surjective\footnote{We use ``injective'' for ``monomorphism'' and ``surjective'' for ``regular epimorphism.''} pieces in a nice way; see Definition~\ref{defn:regcat} for details. Many familiar categories are regular, such as the categories of sets, groups, and rings; also, any abelian category is regular, as is any topos. Let $\cE$ be a regular category that is finitely powered, i.e., every object has finitely many subobjects. Knop introduced a notion \defn{degree function} on $\cE$. Given a degree function $\nu$, he constructed a tensor category $\TT(\cE; \nu)$, and gave sufficient conditions for its semi-simplification to be pre-Tannakian. If $\cE$ is the opposite of the category of finite sets then there is a 1-parameter family $\nu_t$ of degree functions, and $\TT(\cE; \nu_t)$ recovers Deligne's category $\uRep(\fS_t)$.

In recent work with Harman \cite{repst}, we gave another general construction of pre-Tannakian categories. Recall that an \defn{oligomorphic group} is a permutation group $(\fG, \Omega)$ such that $\fG$ has finitely many orbits on $\Omega^n$ for all $n$. Let $\fG$ be an oligomorphic (or even pro-oligomorphic) group. We introduced a certain notion \defn{measure} for $\fG$. Given such a measure $\mu$, we constructed a tensor category $\uPerm(\fG; \mu)$, and gave sufficient conditions for it to admit a pre-Tannakian envelope. If $\fG$ is the infinite symmetric group then there is a 1-parameter family $\mu_t$ of measures, and $\uPerm(\fG; \mu_t)$ recovers $\uRep(\fS_t)$.

The primary purpose of this paper is to show that the tensor categories produced by Knop's construction can all be obtained from the Harman--Snowden construction. Along the way, we obtain some interesting results about regular categories. We also use the picture we develop to produce some interesting examples of measures on oligomorphic groups.

\subsection{Main results}

Fix a finitely-powered regular category $\cE$. For simplicity, we suppose for the moment that $\cE$ is exact and the final object of $\cE$ has no proper subobject. Our goal is to relate Knop's tensor category $\TT(\cE; \nu)$ to one of the Harman--Snowden tensor categories $\uPerm(\fG; \mu)$. The first problem here is to determine the group $\fG$. This is addressed in the following theorem (the various terminology is defined in the body of the paper):

\begin{theorem} \label{mainthm1}
We construct a pro-oligomorphic monoid $\fM$ with the following properties.
\begin{enumerate}
\item There is a fully faithful embedding
\begin{displaymath}
\cE \to \text{\{ $\fM$-sets \}}, \qquad X \mapsto \fA(X).
\end{displaymath}
The essential image consists of smooth, monogenic, and simply connected $\fM$-sets.
\item Let $\cE^s$ be the subcategory of $\cE$ where morphisms are surjections, and let $\fG$ be the group of invertible elements of $\fM$ (which is pro-oligomorphic, by definition). Then there is a fully faithful embedding
\begin{displaymath}
\cE^s \to \text{\{ $\fG$-sets \}}, \qquad X \mapsto \fB(X).
\end{displaymath}
The essential image consists of smooth, transitive, and simply connected $\fG$-sets.
\item The restriction of $\fA(X)$ to $\fG$ decomposes as $\coprod_{Y \subset X} \fB(Y)$.
\end{enumerate}
\end{theorem}

The theorem shows that finitely-powered regular categories can be described by pro-oligomorphic monoids, and raises the interesting possibility of applying results from model theory to study such categories. For instance, suppose that $X$ is a group object in $\cE$. Then $\Gamma=\fA(X)$ is a group on which $\fG$ acts oligomorphically by group automorphisms. It follows that $\Gamma$ is an $\omega$-categorical group, i.e., $\Aut(\Gamma)$ has finitely many orbits on $\Gamma^n$ for all $n$. The structure of such groups is known to be greatly constrained; see, e.g., \cite{Wilson}.

\begin{example}
We give a few examples of the theorem.
\begin{enumerate}
\item Let $\cE$ be the opposite of the category of non-empty finite sets. Then $\fM$ consists of all self maps of a countably infinite set $\Omega$, and $\fG$ is the infinite symmetric group. We have $\fA([n])=\Omega^n$ and $\fB([n])=\Omega^{[n]}$, where $\Omega^{[n]}$ is the subset of $\Omega^n$ consisting of tuples with distint entries.
\item Let $\cE$ be the category of finite sets. This does not exactly fit the set-up of the theorem since the final object has a proper subobject, but the necessary modifications are minor. In this case, $\fM$ consists of all continuous self maps of the Cantor set.
\item Let $\cE$ be the category of finite groups. Then $\fM$ consists of all continuous endomorphisms of the free profinite group $\Omega$ on a countable generating set. We have $\fA(G)=\Hom(\Omega, G)$ and $\fB(G)=\Surj(\Omega, G)$. \qedhere
\end{enumerate}
\end{example}

Now that we have a pro-oligomorphic group $\fG$ associated to $\cE$, we can move on to tensor categories. The following is our main result in this direction.

\begin{theorem} \label{mainthm2}
Given a degree function $\nu$ on $\cE$, there is an associated measure $\mu$ on $\fG$. Moreover, Knop's category $\TT(\cE; \nu)$ is equivalent to a certain subcategory of the Harman--Snowden category $\uPerm(\fG; \mu)$ (after taking the Karoubi envelope).
\end{theorem}

Additionally, we characterize those measures for $\fG$ that arise from a degree function on $\cE$, and we show by example that not every measure arises in this way, in general.

There is an important feature of the equivalence in the theorem that is worth mentioning. In the category $\TT(\cE; \nu)$ the basic objects and morphisms come from the $\fA(X)$ sets, while in the category $\uPerm(\fG; \mu)$ the basic objects and morphisms come from the $\fB(X)$ sets. The equivalence between the two categories does not preserve the ``preferred'' objects and morphisms. The change of basis between the two points of view is upper triangular, as Theorem~\ref{mainthm1}(c) shows.

\subsection{The pre-Galois case}

Let $G$ be an oligomorphic group and let $\cE=\bS(G)$ be the category of finitary smooth $G$-sets. This is a finitely-powered regular category. By Theorem~\ref{mainthm1} (or a mild generalization), there is a pro-oligomorphic group $\fG$ associated to $\cE$. The group $\fG$ is much more complicated than $G$; indeed, to each non-empty smooth $G$-set there is a corresponding \emph{transitive} $\fG$-set. If $G$ is the trivial group then $\fG$ is the homeomorphism group of the Cantor set.

Since $\cE$ is a regular category, it carries a trivial degree function, which induces a measure on $\fG$ by Theorem~\ref{mainthm2}. A simple computation shows that this measure is regular. We show that $\fG$ carries at most one other regular measure. The second regular measure exists precsiely when $G$ carries a regular measure valued in $\bF_2$, and it never comes from a degree function. This is related to a construction from \cite{dblexp}, as we explain.

The above picture is a source for some new examples and counter-examples of measures on oligomorphic groups. For instance, we prove:

\begin{theorem} \label{mainthm3}
There exist oligomorphic groups of arbitrarily fast growth admitting regular measures.
\end{theorem}

See Theorem~\ref{thm:fast} for a precise statement. Prior to this theorem, the fastest growth rate of an oligomorphic group known to admit a regular measure was roughly $\exp(\exp(n^2))$ \cite{dblexp,Kriz}. We note that an important open question is whether there exist pre-Tannakian categories of arbitrarily fast growth rate; see \cite{dblexp}. Unfortunately, the measures produced in our proof of the above theorem do not lead to such categories, as there are nilpotents of non-zero trace in the associated tensor categories.

\subsection{Related work}

While oligomorphic groups have been widely studied, oligomorphic monoids have received comparatively little attention. As far as we can tell, they were first introduced in \cite{MP}, though they were certainly anticipated by \cite{CN}. We note that the definition of oligomorphic monoid used in this paper is far stricter than that of \cite{MP}. The recent paper \cite{MS} discusses non-archimedean topologies on monoids, which is closely related.

In previous work \cite{cantor}, we studied the Harman--Snowden theory in the case where the oligomorphic group is the homeomorphism group of the Cantor set. This corresponds to Knop's theory when $\cE$ is the category of finite sets. Several results from this paper are generalizations of results from that paper.

\subsection*{Acknowledgments}

We thank Nate Harman and Ilia Nekrasov for helpful discussions.

\section{Regular categories} \label{s:regcat}

In this section, we review some material on regular categories.

\subsection{Basic definitions}

Let $\cE$ be a category. We recall two basic notions. First, if $Y \to X$ is a morphism then its \defn{kernel pair} is the diagram $Y \times_X Y \rightrightarrows Y$, assuming this fiber product exists. Second, a \defn{regular epimorphism} $Y \to X$ is a morphism that occurs as the co-equalizer of some diagram $Z \rightrightarrows Y$. Such a morphism is indeed an epimorphism. Moreover, if $\cE$ has fiber products, then a regular epimorphism is the cokernel of its kernel pair.

We now recall the following key concept:

\begin{definition} \label{defn:regcat}
A category $\cE$ is \defn{regular} if it satisfies the following conditions:
\begin{enumerate}
\item It is finitely complete (and in particular admits a final object $\bone$).
\item The kernel pair of any morphism admits a co-equalizer.
\item A base change of a regular epimorphism is a regular epimorphism. \qedhere
\end{enumerate}
\end{definition}

Fix a regular category $\cE$ for the rest of \S \ref{s:regcat}. We use the terms ``injection'' and ``surjection'' in place of ``monomorphism'' and ``regular epimorphism.'' We recall some standard properties of regular categories; see \cite{Barr, Gran} for details. Any morphism $f \colon Y \to X$ factors as $i \circ g$, where $i$ is injective and $g$ is surjective. This factorization is unique (up to isomorphism) and preserved by pull-backs. The image of $f$, denoted $\im(f)$, is the subobject $i$ of $X$. Clearly, $f$ is surjective if and only if $\im(f)=X$. A composition of surjective morphisms is surjective.

\subsection{Relations}

A \defn{relation} from an object $X$ to an object $Y$ is a subobject $A$ of $Y \times X$. If $A$ is a relation from $X$ to $Y$ its \defn{transpose}, denoted $A^t$, is the relation from $Y$ to $X$ obtained by transferring $A$ via the isomorphism $Y \times X = X \times Y$. Suppose $A$ is a relation from $X$ to $Y$, and $B$ is a relation from $Y$ to $Z$. The \defn{composition} of $A$ and $B$, denoted $B \circ A$, is the image of $B \times_Y A$ in $Z \times X$, which is a relation from $X$ to $Z$. It is a basic fact about regular categories that this operation is associative; see \cite[\S 2]{Gran}. The diagonal $\Delta_X \subset X \times X$ is the identity for composition.

A relation $R \subset X \times X$ is an \defn{equivalence relation} if it is reflexive ($\Delta_X \subset R$), symmetric ($R=R^t$), and transitive ($R \circ R \subset R$). The kernel pair $R$ of a morphism $f \colon Y \to X$ is an equivalence relation on $Y$ \cite[Lemma~2.2]{Gran}, and if $f$ is surjective then it is the cokernel of $R$. An equivalence relation is called \defn{effective} if it is the kernel pair of some morphism. 

\subsection{Exact categories} \label{ss:exact}

An \defn{exact category} is a regular category in which all equivalence relations are effective. Every regular category $\cE$ admits an \defn{exact completion}, denoted $\cE_{\rm ex/reg}$. The category $\cE_{\rm ex/reg}$ is exact (of course), and comes with a fully faithful functor $\cE \to \cE_{\rm ex/reg}$ that preserves finite limits and surjections; moreover, it is the universal such category. Every object of $\cE_{\rm ex/reg}$ is a quotient of an object in $\cE$ by an equivalence relation from $\cE$. A simple construction of the exact completion is given in \cite{Lack}, and the properties we assert are proved there as well. Due to the nice properties of the exact completion, we will typically not lose generality by assuming our categories are exact.

\subsection{Finitely-powered categories} \label{ss:fp}

We say that a category is \defn{finitely-powered} if every object has finitely many subobjects. Let $\cE$ be a finitely-powered regular category. We make a few observations.

(a) A morphism $f \colon X \to Y$ is determined by its graph $\Gamma_f \subset X \times Y$, which is defined to be the image of the map $(\id,f) \colon X \to X \times Y$. It follows that $\Hom(X,Y)$ injects into $\Sub(X \times Y)$, and is therefore a finite set.

(b) By a \defn{quotient} of $X$, we mean a surjection $X \to Y$. As we have seen, a quotient of $X$ is determined (up to isomorphism) by its kernel pair, which is a subobject of $X \times X$. Since there are finitely many such subobjects, we see that $X$ has finitely many quotients (up to isomorphism).

(c) The exact completion $\cE_{\rm ex/reg}$ is also finitely powered. To see this, first suppose that $X$ is an object of $\cE$ and $Y$ is a subobject of $X$ in $\cE_{\rm ex/reg}$. Then $Y$ is a quotient of an object $Y'$ from $\cE$, and thus the image of a map $f \colon Y' \to X$ in $\cE_{\rm ex/reg}$. Since $\cE \to \cE_{\rm ex/reg}$ is fully faithful and preserves images, we see that $f$ and its image come from $\cE$. Thus the subobjects of $X$ in $\cE_{\rm ex/reg}$ are the same as those in $\cE$, and in particular finite in number. If $X$ is now an arbitrary object of $\cE_{\rm ex/reg}$ then there is a surjection $X' \to X$, where $X'$ comes from $\cE$, and pull-back gives an injection $\Sub(X) \to \Sub(X')$. Thus $\Sub(X)$ is finite, as required.

\section{Oligomorphic groups and monoids}

In this section we review some material on oligomorphic groups and related objects.

\subsection{Oligomorphic groups} \label{ss:oliggp}

We begin by recalling some definitions:

\begin{definition}
An \defn{oligomorphic group} is a permutation group $(\fG, \Omega)$ such that $\fG$ has finitely many orbits on $\Omega^n$ for all $n \ge 0$.
\end{definition}

\begin{definition}
A \defn{pro-oligomorphic group}\footnote{In some references, these are called \defn{admissible groups}, but we prefer this more descriptive term.} is a topological group $\fG$ that is Hausdorff, non-archimedean (open subgroups form a neighborhood basis of the identity), and Roelcke-precompact (if $U$ and $V$ are open subgroups then $U \backslash G/V$ is finite).
\end{definition}

If $\fG$ is an oligomorphic group then it carries a pro-oligomorphic topology in which the basic open subgroups are pointwise stabilizers of finite subsets of $\Omega$ \cite[\S 2.2]{repst}. Conversely, if $\fG$ is pro-oligomorphic then for each open subgroup $U$ of $\fG$ the image of $\fG$ in the group of permutations of $\fG/U$ is oligomorphic, and $\fG$ is a dense subgroup of the inverse limit of these oligomorphic groups. For the purposes of this paper, there is no difference between a group and a dense subgroup, so all pro-oligomorphic groups can be taken to be an inverse limit of oligomorpic groups. We refer to \cite{CameronBook} for general background on oligomorphic groups.

Let $\fG$ be a pro-oligomorphic group. An action of $\fG$ on a set $X$ is \defn{smooth} if every point of $X$ has open stabilizer, and \defn{finitary} if $\fG$ has finitely many orbits on $X$. We write $\hat{\bS}(\fG)$ and $\bS(\fG)$ for the categories of smooth and finitary smooth sets. A finite product of (finitary) smooth sets is again (finitary) smooth; see \cite[\S 2.3]{repst}.

\subsection{Pre-Galois categories}

In \cite{repst}, we defined tensor categories associated to a pro-oligomorphic group $\fG$. In fact, the constructions of that paper are most naturally in terms of the category $\bS(\fG)$ rather than $\fG$ itself. We therefore set out to characterize these categories intrinsically. This led to the following notion:

\begin{definition} \label{defn:pregalois}
An essentially small category $\cC$ is \defn{pre-Galois} if the following hold:
\begin{enumerate}
\item $\cC$ has finite co-products (and thus an initial object $\bzero$).
\item Every object of $\cC$ is isomorphic to a finite co-product of \defn{atoms}, i.e., objects that do not decompose under co-product.
\item If $X$ is an atom and $Y$ and $Z$ are other objects, then any map $X \to Y \amalg Z$ factors uniquely through $Y$ or $Z$.
\item $\cC$ has fiber products and a final object $\bone$.
\item Any monomorphism of atoms is an isomorphism.
\item If $X \to Z$ and $Y \to Z$ are maps of atoms then $X \times_Z Y$ is non-empty (i.e., not $\bzero$).
\item The final object $\bone$ is atomic.
\item Equivalence relations in $\cC$ are effective. \qedhere
\end{enumerate}
\end{definition}

One easily sees that $\bS(\fG)$ is pre-Galois if $\fG$ is pro-oligomorphic. The following is the main theorem of \cite{pregalois}:

\begin{theorem}
A category is pre-Galois if and only if it is equivalent to $\bS(\fG)$ for some pro-oligomorphic group $\fG$.
\end{theorem}

It is at times useful to consider a generalization of the above picture.

\begin{definition}
Let $\cC_0$ be a pre-Galois category. An \defn{order}\footnote{We used the term ``non-degenerate $\rB$-category'' for these categories in \cite{pregalois}.} in $\cC_0$ is a full subcategory $\cC$ that satisfying the following conditions:
\begin{enumerate}
\item $\cC$ contains the final object $\bone$
\item $\cC$ is closed under finite products, finite co-products, and passing to subobjects
\item Every object of $\cC_0$ is a quotient of an object in $\cC$. \qedhere
\end{enumerate}
\end{definition}

Note that an order is determined by the atoms it contains. If $X \to Z$ and $Y \to Z$ are maps in an order $\cC$ then the fiber product $X \times_Z Y$ (taken in $\cC_0$) is a subobject of $X \times Y$, and thus belongs to $\cC$; thus $\cC$ is closed under fiber products.

A \defn{stabilizer class} in a pro-oligomorphic group $\fG$ is a collection $\fU$ of open subgroups of $\fG$ satisfying the following conditions: (a) $\fU$ is closed under finite intersections; (b) $\fU$ is closed under conjugation; (c) $\fG$ belongs to $\fU$; (d) every open subgroup of $\fG$ contains a member of $\fU$ as a subgroup. Given a stabilizer class, we say that an action of $G$ on a set $X$ is \defn{$\fU$-smooth} if the stabilizer of every point is contained in $\fU$. We let $\bS(\fG; \fU)$ be the category of finitary smooth $\fU$-sets. One easily sees that this category is an order in $\bS(\fG)$, and this is the motivating example of the concept. In \cite{pregalois}, we also proved the following theorem:

\begin{theorem} \label{thm:order}
Let $\cC$ be a category. The following are equivalent:
\begin{enumerate}
\item $\cC$ satisfies axioms (a)--(g) of Definition~\ref{defn:pregalois}.
\item $\cC$ is an order in some pre-Galois category $\cC_0$.
\item $\cC$ is equivalent to $\bS(\fG; \fU)$ for some pro-oligomorphic group $\fG$ and stabilizer class $\fU$.
\end{enumerate}
\end{theorem}

\begin{remark}
The category $\cC_0$ in (b) is unique, up to equivalence. It can be obtained canonically from $\cC$ as its exact completion.
\end{remark}

\begin{example}
Let $\fS$ be the infinite symmetric group acting on $\Omega=\{1,2,3,\ldots\}$. This is the most basic example of an oligomorphic group. For an integer $n \ge 1$, let $\fS(n)$ be the subgroup of $\fS$ stabilizing each of $1, \ldots, n$. A subgroup of $\fS$ is open if and only if it contains some $\fS(n)$. In fact, every open subgroup of $\fS$ is conjugate to one of the form $H \times \fS(n)$, where $H$ is a subgroup of the finite symmetric group $\fS_n$ \cite[Proposition~15.1]{repst}.

The class $\fU$ of subgroups conjugate to some $\fS(n)$ is a stabilizer class. The transitive $\fU$-smooth sets are exactly the orbits appearing in $\Omega^n$ for $n \ge 0$. An example of a set that is smooth but not $\fU$-smooth is the set of 2-element subsets of $\Omega$; this is isomorphic to $\fS/U$, where $U=\fS_2 \times \fS(2)$. In the category $\bS(\fS; \fU)$, the equivalence relation on $\Omega^2$ induced by the $\fS_2$ action is not effective.
\end{example}

\subsection{Atomic categories} \label{ss:atomic}

Define the $\amalg$-envelope of a category $\cA$ to be the category whose objects are formal finite coproducts of objects in $\cA$; see \cite[\S 5.1]{pregalois}. If $\cC$ is an order in a pre-Galois category, then one easily sees that $\cC$ is the $\amalg$-envelope of its category of atoms. It is therefore a natural problem to determine which categories occur as categories of atoms in pre-Galois categories. We solved this problem in \cite{pregalois}.

To state the answer, we need to introduce some terminology. Let $\cA$ be a category. A \defn{final set} for $\cA$ is a set of objects $\{X_i\}_{i \in I}$ such that for each object $Y$ the set $\coprod_{i \in I} \Hom(Y, X_i)$ has cardinality one; that is, $Y$ maps to a unique $X_i$, and it maps to it uniquely. Let $X \to Z$ and $Y \to Z$ be morphisms in $\cA$. Consider the factor of objects $T$ equipped with morphisms $T \to X$ and $T \to Y$ making the square commutes. A final set for this category is called a \defn{fiber product set}.

\begin{definition}
An \defn{atomic category} is a category $\cA$ satisfying the following conditions:
\begin{enumerate}
\item The category $\cA$ has a final object.
\item Given maps $X \to Z$ and $Y \to Z$, there is a finite non-empty fiber product set.
\item Every monomorphism in $\cA$ is an isomorphism. \qedhere
\end{enumerate}
\end{definition}

The following is the main result we are after:

\begin{proposition}
If $\cC$ is an order in a pre-Galois category then its category of atoms is an atomic category. Conversely, if $\cA$ is an atomic category then its $\amalg$-envelope is an order in a pre-Galois category.
\end{proposition}

\begin{proof}
This is essentially \cite[Theorem~5.8]{pregalois} $\cA$ is atomic if and only if (in the terminology of \cite{pregalois}) $\cA^{\op}$ is an $\rA$-category with an initial object and the amalgamation property.
\end{proof}

Let $\cA$ be an atomic category and let $\cC$ be its $\amalg$-envelope. By Theorem~\ref{thm:order}, $\cC$ is equivalent to $\bS(\fG; \fU)$ for some pro-olgiomorphic group $\fG$ and stabilizer class $\fU$. We now describe how to directly obtain $\fG$ and $\fU$, as this will be useful to us. We assume for simplicity that $\cA$ has countably many isomorphism classes.

Let $\Omega=(\Omega_i)_{i \in \bN}$ be a pro-object in $\cA$; see \cite[Appendix~A]{homoten} for background on pro-objects. We make a few definitions:
\begin{enumerate}
\item $\Omega$ is \defn{universal} if for every object $X$ of $\cA$ there is a map $\Omega \to X$.
\item $\Omega$ is \defn{homogeneous} if for every object $X$ the group $\Aut(\Omega)$ acts transitively on $\Hom(\Omega, X)$.
\item $\Omega$ is \defn{f-projective} if for morphisms $f \colon \Omega \to X$ and $h \colon Y \to X$, there exists a morphism $g \colon \Omega \to Y$ such that $f=h \circ g$; here $X$ and $Y$ are objects of $\cA$.
\end{enumerate}
See \cite[Appendix~A]{homoten} for general background on these concepts, though note that all arrows are reversed there. A universal pro-object is homogeneous if and only if it is f-projective by \cite[Proposition~A.7]{homoten}. If $\cA$ is an atomic category then there is a universal homogeneous pro-object $\Omega$, which is unique up to isomorphism; see \cite[Appendix~A]{homoten} and \cite[\S 6.3]{pregalois}.

Put $\fG=\Aut(\Omega)$. For an object $X$ of $\cA$, consider the $\fG$-set $P(X)=\Hom(\Omega, X)$. This set is non-empty (since $\Omega$ is universal) and transitive (since $\Omega$ is homogeneous). Let $\fU$ be the set of subgroups of $\fG$ that stabilizer a map $\Omega \to X$, for $X$ in $\cA$. We endow $\fG$ with the topology generated by $\fU$. With this topology, $\fG$ is pro-oligomorphic, $\fU$ is a stabilizer class, and the $P(X)$ are exactly the transitive smooth $\fG$-sets.

In general (without hypothesis on the number of isomorphism classes in $\cA$), the above claims still hold, though one must use pro-objects indexed by arbitrary directed sets. For instance, we know that $\cA$ is equivalent to a category of the form $\bS(\fG'; \fU')$, and one can take $\Omega$ to be the pro-object $(\fG'/U)$, as $U$ varies over the open subgroups in $\fU'$.

\subsection{Finite covers}

Let $X$ be a transitive smooth $\fG$-set. A \defn{finite cover} is a $\fG$-equivariant map $f \colon Y \to X$ where $Y$ is smooth and $f$ has finite fibers. We say $f$ is \defn{trivial} $Y$ is isomorphic to $X^{\amalg n}$ for some $n$, as a $\fG$-set over $X$. We say that $X$ is \defn{simply connected} if every finite cover is trivial. We require the following general observation:

\begin{proposition} \label{prop:triv-cover}
Suppose that $Y \to X$ is a finite cover, with $X$ transitive. Then there exists a map $X' \to X$, with $X'$ transitive and smooth, such that the base change $Y' \to X'$ is a trivial cover.
\end{proposition}

\begin{proof}
It suffices to treat the case where $Y$ itself is transitive. Write $X=\fG/U$ and $Y=\fG/V$ where $V \subset U$ is a finite index containment of open subgroups. Let $U'$ be the intersection of all conjugates of $V$ in $U$. We can then take $X'=\fG/U'$.
\end{proof}

\subsection{Oligomorphic monoids} \label{ss:monoid}

We now extend the pro-oligomorphic concept to monoids. For this, we require a notion of non-archimedean monoid. This can be defined in topological terms\footnote{See \cite{MS} for this approach.}; however, since we only use the topology to speak of smooth sets, it will be more convenient to work directly with this concept:

\begin{definition} \label{defn:nonarch}
A \defn{non-archimedean structure} on a monoid $\fM$ is a class $\hat{\bS}(\fM)$ of $\fM$-sets, called the \defn{smooth} $\fM$-sets, satisfying the following conditions:
\begin{enumerate}
\item A subquotient of a smooth $\fM$-set is smooth.
\item If $X$ is an $\fM$-set and $\{Y_i\}_{i \in I}$ is a family of smooth $\fM$-subsets then $\bigcup_{i \in I} Y_i$ is smooth.
\item A finite product of smooth $\fM$-sets is smooth.
\end{enumerate}
A \defn{non-archimedean monoid} is a monoid equipped with a non-archimedean structure. We say that such a monoid $\fM$ is is \defn{Hausdorff} if the following additional condition holds:
\begin{enumerate}
\setcounter{enumi}{3}
\item Given $\sigma \ne \tau$ in $\fM$, there exists a smooth $\fM$-set $X$ and $x \in X$ such that $\sigma x \ne \tau x$. \qedhere
\end{enumerate}
\end{definition}

Given a monoid $\fM$ and some class $\Sigma$ of $\fM$-sets, there is a minimal non-archimedean structure $\hat{\bS}(\fM)$ containing $\Sigma$; we say that $\Sigma$ \defn{generates} $\hat{\bS}(\fM)$. If $f \colon \fM \to \fM'$ is a monoid homomorphism and we have a non-archimedean structure on $\fM'$, then there is an induced structure on $\fM$, namely, the one generated by restrictions of smooth $\fM'$-sets. If $\fM'$ is Hausdorff and $f$ is injective then the induced structure on $\fM$ is Hausdorff. Given a non-archimedean monoid $\fM$, we write $\bS(\fM)$ for the class of finitary smooth $\fM$-sets, where an $\fM$-set $X$ is finitary if there exist $x_1, \ldots, x_n$ such that $X=\bigcup_{i=1}^n \fM x_i$.

For a monoid $\fM$, we let $\fM^{\circ}$ denote the group of invertible elements in $\fM$.

\begin{definition}
A \defn{pro-oligomorphic monoid} is a Hausdorff non-archimedean monoid $\fM$ such that $\fM^{\circ}$ is pro-oligomorphic (with the induced non-archimedean structure) and finitary smooth $\fM$-sets restrict to finitary $\fM^{\circ}$-sets.
\end{definition}

This definition is very restrictive, since it requires $\fM$ to contain a very large group. However, it is exactly what we will require in this paper. See \cite{MP} for a much less restrictive notion of oligomorphic monoid.

\subsection{Oligomorphic categories}

Let $\fM$ be a category with a finite set of objects $\Lambda$. An \defn{$\fM$-set} is a functor $\fM \to \Set$. Although we use the language of categories, we are actually thinking in more concrete terms. An $\fM$-set $\fZ$ is just a finite collection of sets $(\fZ_{\lambda})_{\lambda \in \Lambda}$, together with action maps
\begin{displaymath}
\Hom_{\fM}(\lambda, \mu) \times \fZ_{\lambda} \to \fZ_{\mu}
\end{displaymath}
satisfying some conditions. Thus $\fM$-sets are just a mild generalization of a set equipped with an action of a monoid.

We define a non-archimedean structure on $\fM$ in the same way as the monoid case: one gives a collection $\hat{\bS}(\fM)$ of $\fM$-sets as in Definition~\ref{defn:nonarch}. We note that a non-archimedean structure on $\fM$ induces one on $\End(\lambda)$ and $\Aut(\lambda)$ for all $\lambda \in \Lambda$. We again have a notion of finitary set, and write $\bS(\fM)$ for the class of finitary smooth $\fM$-sets. We will require the following mild generalization of pro-oligomorphic monoid:

\begin{definition}
A \defn{pro-oligomorphic category} is a Hausdorff non-archimedean category $\fM$ with the following properties:
\begin{enumerate}
\item $\fM$ has finitely many objects.
\item $\fG_{\lambda}=\Aut(\lambda)$ is a pro-oligomorphic, for all objects $\lambda$.
\item If $\fZ$ is a finitary smooth $\fM$-set then $\fZ_{\lambda}$ is a finitary $\fG_{\lambda}$-set for all $\lambda$. \qedhere
\end{enumerate}
\end{definition}

\section{The first embedding theorem} \label{s:embed1}

The purpose of this section is to prove a statement like Theorem~\ref{mainthm1}(b), which connects the surjection category of a finitely-powered regular category to oligomorphic groups.

\subsection{The surjection category}

Fix a finitely-powered regular category $\cE$ throughout \S \ref{s:embed1}. Let $\{E_{\lambda}\}_{\lambda \in \Lambda}$ be the subobjects of the final object $\bone$ in $\cE$, where $\Lambda$ is a finite index set. We say that an object $X$ of $\cE$ has \defn{type $\lambda$} if the image of the natural map $X \to \bone$ is $E_{\lambda}$. Note that if $Y \to X$ is a surjective map then $X$ and $Y$ have the same type.

Let $\cE^s$ be the category with the same ojects as $\cE$, and whose morphisms are surjections. For $\lambda \in \Lambda$, let $\cE^s_{\lambda}$ be the full subcategory of $\cE^s$ spanned by objects of type $\lambda$. By the above comments, there are no morphisms in $\cE^s$ between objects of different types, and so $\cE^s$ is the disjoint union of the categories $\cE^s_{\lambda}$ for $\lambda \in \Lambda$.

\subsection{Construction of groups}

We now construct some pro-oligomorphic groups associated to $\cE$. Throughout our discussion, all fiber products are taken in $\cE$. We say that a subobject of a (fiber) product is \defn{ample} if it surjects onto each factor.

\begin{proposition}
The category $\cE^s_{\lambda}$ is atomic.
\end{proposition}

\begin{proof}
(a) The final object of $\cE_{\lambda}^s$ is $E_{\lambda}$.

(b) Consider maps $X \to Z$ and $Y \to Z$ in $\cE^s_{\lambda}$. We claim that the ample subobjects of $X \times_Z Y$ are a fiber product set in $\cE^s_{\lambda}$. Indeed, let $T \to X$ and $T \to Y$ be maps in $\cE^s_{\lambda}$ such that the natural square commutes. Then we obtain a canonical map $f \colon T \to X \times_Y Z$ in $\cE$. We have a unique (up to isomorphism) factorization of $f$ as
\begin{displaymath}
\xymatrix{
T \ar[r]^-g & W \ar[r]^-i & X \times_Z Y }
\end{displaymath}
where $g$ is surjective and $i$ is injective. Since $T$ surjects onto $X$ and $Y$, it follows that $W$ is an ample subobject. Since $g$ is a morphism in $\cE^s_{\lambda}$, the claim follows. The set of ample subobjects is finite (since $\cE$ is finitely-powered) and non-empty (it contains the fiber product).

(c) Let $f \colon X \to Y$ be a monomorphism in $\cE^s_{\lambda}$. Let $p_i \colon X \times_Y X \to X$, for $i=1,2$, be the projections. These are surjections, as they are obtained from $f$ by base change, and so they are morphisms in $\cE^s_{\lambda}$. Since $fp_1=fp_2$ and $f$ is a monomorphism in $\cE^s_{\lambda}$, we have $p_1=p_2$. It follows that $f$ is a monomorphism in $\cE$. Thus $f$ is an isomorphism in $\cE$, and therefore in $\cE^s_{\lambda}$ as well.
\end{proof}

We now apply the theory from \S \ref{ss:atomic}. Let $\Omega_{\lambda}$ be the universal homogeneous pro-object in $\cE^s_{\lambda}$, and let $\fG_{\lambda}=\Aut(\Omega_{\lambda})$. Given an object $X$ of type $\lambda$, we let
\begin{displaymath}
\fB(X) = \Surj_{\Pro(\cE)}(\Omega_{\lambda}, X)
\end{displaymath}
be the associated transitive $\fG_{\lambda}$-set. The $\fB(X)$'s define a non-archimedean topology on $\fG_{\lambda}$, and make $\fG_{\lambda}$ into a pro-oligomorphic group. We note that, by construction, we have a fully faithful functor
\begin{displaymath}
\cE^s_{\lambda} \to \bS(\fG_{\lambda}), \qquad X \mapsto \fB(X).
\end{displaymath}
We record the following useful piece of information:

\begin{proposition} \label{prop:B-prod}
Consider morphisms $X \to Z$ and $Y \to Z$ in $\cE^s_{\lambda}$. Then we have a natural isomorphism of $\fG_{\lambda}$-sets
\begin{displaymath}
\fB(X) \times_{\fB(Z)} \fB(Y) = \coprod \fB(W),
\end{displaymath}
where the union is taken over ample subobjects $W$ of $X \times_Z Y$.
\end{proposition}

\begin{proof}
It is a general fact that the coproduct of the fiber product set in an atomic category becomes the actual fiber product in the associated pre-Galois category \cite[Proposition~5.5]{pregalois}. One can also see this isomorphism more directly using the homogeneous property of $\Omega_{\lambda}$.
\end{proof}

\begin{remark}
Consider the exact completion $\cE_{\rm ex/reg}$ of $\cE$ (\S \ref{ss:exact}). The subobjects of $\bone$ in $\cE$ and $\cE_{\rm ex/reg}$ are the same (see \S \ref{ss:fp}(c)). Moreover, $\cE_{\lambda}^s$ is cofinal in $(\cE_{\rm ex/reg})_{\lambda}^s$, since every object of $\cE_{\rm ex/reg}$ is a quotient of an object from $\cE$. It follows that $\Omega_{\lambda}$ is a universal homogeneous object for $(\cE_{\rm ex/reg})_{\lambda}^s$. We thus see that $\cE$ and $\cE_{\rm ex/reg}$ lead to the same system $(\fG_{\lambda})_{\lambda \in \Lambda}$ of pro-oligomorphic groups.
\end{remark}

\subsection{Classification of transitive sets} \label{ss:trans}

Given an oligomorphic group constructed as the automorphism group of a Fra\"iss\'e limit, it is generally an important (and sometimes difficult) problem to classify all open subgroups. We now solve this problem for the groups associated to $\cE$. We assume that $\cE$ is exact in \S \ref{ss:trans}.

\begin{proposition} \label{prop:Gclass}
Every transitive smooth $\fG_{\lambda}$-set has the form $\fB(X)/\Gamma$ for some $X \in \cE^s_{\lambda}$ and $\Gamma \subset \Aut(X)$.
\end{proposition}

As a corollary, we obtain a classification of the open subgroups of $\fG_{\lambda}$. For an object $X$ of type $\lambda$, we let $\fG_{\lambda}(X)$ be the stabilizer in $X$ of a chosen surjection $\Omega_{\lambda} \to X$.

\begin{corollary} \label{cor:Gclass}
Let $U$ be an open subgroup of $\fG_{\lambda}$. Then there is some object $X$ of type $\lambda$ such that, after replacing $U$ by a conjugate subgroup, we have $\fG_{\lambda}(X) \subset U \subset \rN\fG_{\lambda}(X)$.
\end{corollary}

We note that the above proposition is \cite[Proposition~3.10]{cantor} in the case where $\cE$ is the category of finite sets. The present proof is modeled on the proof there.

The proof of the proposition will take the remainder of \S \ref{ss:trans}. By definition, every transitive smooth $\fG_{\lambda}$-set is a quotient of some $\fB(X)$. To prove the proposition, we study equivalence relations on $\fB(X)$. Let $\cA$ denote the of ample subobjects of $X \times X$. Recall from Proposition~\ref{prop:B-prod} that the $\fG_{\lambda}$-orbits on $\fB(X) \times \fB(X)$ are indexed by $\cA$, and so $\fG_{\lambda}$-stable subsets of $\fB(X) \times \fB(X)$ correspond to subsets of $\cA$. As a first step, we characterize equivalence relations on $\fB(X)$ from this point of view.

\begin{lemma} \label{lem:Gclass-1}
Let $\fR$ be a $\fG_{\lambda}$-stable subset of $\fB(X) \times \fB(X)$ corresponding to $\cR \subset \cA$. Then $\fR$ is an equivalence relation if and only if $\cR$ satisfies the following conditions:
\begin{enumerate}
\item $\cR$ contains the diagonal $\Delta_X$.
\item If $A \in \cR$ then $A^t \in \cR$.
\item If $W$ is a subobject of $X \times X \times X$ such that $p_{12}(W)$ and $p_{23}(W)$ belong to $\cR$ then $p_{13}(W)$ belongs to $\cR$.
\end{enumerate}
Moreover, (c) implies $\cR$ is closed under composition, i.e., if $A,B \in \cR$ then $B \circ A \in \cR$.
\end{lemma}

\begin{proof}
By definition, $\fB(X)$ is the set of surjective maps $\Omega_{\lambda} \to X$. Thus $\fB(X) \times \fB(X)$ consists of pairs $(f,g)$ of such maps. If $Y$ is an ample subobject of $X \times X$ then $\fB(Y)$ consists of those pairs $(f,g)$ for which the image of the map $(f,g) \colon \Omega_{\lambda} \to X \times X$ is $Y$.

Now, suppose $\fR$ is an equivalence relation. Then $\fR$ contains $(f,f)$ for any $f$, and so (a) holds. If $(f,g) \in \fR$ then $(g,f) \in \fR$, which shows that (b) holds. Finally, suppose $W$ as in (c) is given. Choose a surjection $\Omega_{\lambda} \to W$, which we identify with a triple $(f,g,h)$ of maps $\Omega_{\lambda} \to X$. The image of $(f,g)$ is $p_{12}(W)$, which belongs to $\cR$, and so $(f,g) \in \fR$. Similarly, $(g,h) \in \fR$. Since $\fR$ is transitive, it follows that $(f,h) \in \fR$. Thus its image, which is $p_{13}(W)$, belongs to $\cR$, and so (c) holds.

Now suppose $\cR$ satisfies (a)--(c). It is clear that $\fR$ is reflexive and symmetric. We now prove that $\fR$ is transitive. Thus let $(f,g)$ and $(g,h)$ belong to $\fR$. Let $W \subset X \times X \times X$ be the image of $(f,g,h)$. Then $p_{12}(W)$ is the image of $(f,g)$, which belongs to $\cR$. Similarly, $p_{23}(W)$ belongs to $\cR$. Thus $p_{13}(W)$ belongs to $\cR$ too. Since this is the image of $(f,h)$, we see that $(f,h)$ belongs to $\fR$, and so $\fR$ is transitive.

Finally, suppose (c) holds, and let $A,B \in \cR$. Let $W=B \times_X A$, which is a subobject of $X \times X \times X$. We have $p_{12}(W)=B$ (since $A$ is ample) and $p_{23}(W)=A$ (since $B$ is ample). Thus $p_{13}(W)=B \circ A$ belongs to $\cR$, as required.
\end{proof}

For $\cR \subset \cA$, let  $\cR^{\circ}$ be the subset of $\cR$ consisting of those relations that contain $\Delta_X$.

\begin{lemma} \label{lem:Gclass-3}
Let $\cR \subset \cA$ satisfy conditions (a)--(c) of Lemma~\ref{lem:Gclass-1}, and let $R$ be a maximal element of $\cR^{\circ}$ (ordered by inclusion). Then:
\begin{enumerate}
\item $R$ is an equivalence relation on $X$.
\item $R$ contains every member of $\cR^{\circ}$.
\item $\cR$ contains every ample subobject of $R$.
\end{enumerate}
\end{lemma}

\begin{proof}
Let $S$ be an arbitrary element of $\cR^{\circ}$. We have a map $R \to R \times_X S$ given by $(x,y) \mapsto ((x,y),(y,y))$, since $S$ contains $\Delta_X$, which shows that $R \circ S$ contains $R$; similarly, it contains $S$. Since $R$ is maximal, it follows that $R=R \circ S$, and $S \subset R$. In particular, $R=R^t$ and $R=R \circ R$, and so $R$ is an equivalence relation. This proves (a) and (b).

We now prove (c). Let $A$ an ample subobject of $R$. Let $W \subset X \times X \times X$ be $p_{12}^{-1}(R) \cap p_{13}^{-1}(A) \cap p_{23}^{-1}(R)$; intuitively, $W$ consists of triples $(x,y,z)$ such that $(x,y) \in R$ and $(y,z) \in R$ and $(x,z) \in A$. We claim $p_{12}(W)=R$. To see this, we use elemental notation (which can be made rigorous using a functor of points approach). Thus suppose $(x,y) \in R$. There is some $z$ such that $(x,z) \in A$ (since $A$ is ample). Since $A \subset R$, it follows that $(x,z) \in R$, and since $R$ is an equivalence relation, we have $(y,z) \in R$. Thus $(x,y,z) \in W$, and so $(x,y) \in p_{12}(W)$. This proves the claim. Similar reasoning shows $p_{23}(W)=R$ and $p_{13}(W)=A$, and so $A \in \cR$ as required.
\end{proof}

Suppose $\Gamma$ is a subgroup of $\Aut(X)$. Then $\Gamma$ acts on $\fB(X)$, and thus induces an equivalence relation $\fR_{\Gamma}$ on it. In the next two lemmas, we establish an important dichotomy involving these relations.

\begin{lemma} \label{lem:Gclass-2}
Let $A$ be an ample subobject of $X \times X$, and let $B=A \circ A^t$. Then $B$ contains $\Delta_X$, and $B=\Delta_X$ if and only if $p_2 \colon A \to X$ is an isomorphism.
\end{lemma}

\begin{proof}
Consider the following cartesian diagram
\begin{displaymath}
\xymatrix@C=3em{
A \ar[r] \ar[d]_{p_1} & A \times_X A^t \ar[d] \\
X \ar[r]^-{\Delta} & X \times X. }
\end{displaymath}
The top map is $(x,y) \mapsto ((x,y),(y,x))$, and the right map is $((x,y),(y,z)) \mapsto (x,z)$. Since $p_1$ is surjective, it follows that $B$ contains $\Delta_X$. Now, in any regular category, if $f \colon X \to Y$ is a map then $X \times_Y \im(f) \to X$ is an isomorphism. Thus if $B=\Delta_X$ then $A \to A \times_X A^t$ is an isomorphism, from which it follows that $p_2 \colon A \to X$ is monic, and thus an isomorphism (since it is also surjective).
\end{proof}

\begin{lemma}
Let $\fR \subset \fB(X) \times \fB(X)$ be an equivalence relation. Then one of the following is true:
\begin{enumerate}
\item There exists a proper surjection $X \to Y$ such that $\fR$ contains $\fB(X) \times_{\fB(Y)} \fB(X)$.
\item $\fR=\fR_{\Gamma}$ for some subgroup $\Gamma$ of $\Aut(X)$.
\end{enumerate}
\end{lemma}

\begin{proof}
Let $\cR \subset \cA$ correspond to $\fR$. Let $R$ be the maximal element of $\cR^{\circ}$, which is an equivalence relation on $X$ by Lemma~\ref{lem:Gclass-3}(a), and let $Y=X/R$. Let $\fR'=\fB(X) \times_{\fB(Y)} \fB(X)$, and let $\cR' \subset \cA$ correspond to $\fR'$. By Proposition~\ref{prop:B-prod}, $\cR'$ is exactly the set of ample subobjects of $X \times_Y X=R$. Thus $\cR'$ is contained in $\cR$ by Lemma~\ref{lem:Gclass-3}(c), and so $\fR'$ is contained in $\fR$. We are therefore in case (a), provided $R$ is larger than $\Delta_X$.

Assume now that $R=\Delta_X$. This means $\Delta_X$ is the only element of $\cR$ containing the diagonal. By Lemma~\ref{lem:Gclass-2} we see that if $A \in \cR$ then $p_2 \colon A \to X$ is an isomorphism; since $A^t \in \cR$ too, we see that $p_1 \colon A \to X$ is also an isomorphism. Thus $A$ is the graph of an automorphism of $X$. Let $\Gamma$ be the collection of these automorphisms. Since $\cR$ is closed under transpose and composition, it follows that $\Gamma$ is a subgroup of $\Aut(X)$. We have $\fR=\fR_{\Gamma}$, as required.
\end{proof}

\begin{proof}[Proof of Proposition~\ref{prop:Gclass}]
Let $\fZ$ be a transitive smooth $\fG_{\lambda}$-set. By definition, there is a surjection $\fB(X) \to \fZ$ for some $X$. Choose $X$ minimal, in the sense that there is no surjection $\fB(W) \to \fZ$ for any proper quotient of $W$ of $X$; this is possible since $X$ has finitely many quotients by \S \ref{ss:fp}(b). We have $\fZ = \fB(X)/R$, where $R$ is the equivalence relation on $\fB(X)$ induced by the surjecttion $\fB(X) \to \fZ$. By our minimality assumption, we are in case~(b) of the previous lemma, which shows that $\fZ=\fB(X)/\Gamma$ for some subgroup $\Gamma$ of $\Aut(X)$. This completes the proof.
\end{proof}

\subsection{Simply connected sets}

Recall from \S \ref{ss:oliggp} the notion of simply connected $\fG_{\lambda}$-set. The next result gives a complete characterization of these sets.

\begin{proposition} \label{prop:sc}
If $X$ is an object of $\cE^s_{\lambda}$ then $\fB(X)$ is simply connected. If $\cE$ is exact, then every simply connected transitive smooth $\fG_{\lambda}$-set is of the form $\fB(X)$.
\end{proposition}

\begin{proof}
It suffices to treat the case where $\cE$ is exact, so we assume this. First suppose that we have a finite cover $\fB(Y) \to \fB(X)$, induced by a surjection $f \colon Y \to X$. Let $\fB(X') \to \fB(X)$ be a map trivializing the cover (Proposition~\ref{prop:triv-cover}), so that $\fB(X') \times_{\fB(Y)} \fB(X)$ is a disjoint union of copies of $\fB(X')$. On the other hand, this fiber product is the disjoint union of $\fB(W)$ over the ample subobjects of $Y'=X' \times_Y X$. We thus see that each such $W$ maps isomorphically to $X'$. Taking $W=Y'$, we see that $Y' \to X'$ is an isomorphism, and so $Y \to X$ is an isomorphism by following lemma. Thus the original cover is trivial.

Now suppose that we have a general finite cover $\fZ \to \fB(X)$. By the classification of transitive sets, we have $\fZ=\fB(Y)/\Gamma$ for some object $Y$ and some subgroup $\Gamma$ of $\Aut(Y)$. Since $\fB(Y) \to \fZ$ is a finite cover, we see that $\fB(Y) \to \fB(X)$ is a finite cover, and so $Y \cong X$. It follows that $\Gamma$ is trivial, and so the cover is trivial.

Finally, suppose that $\fZ$ is an arbitrary simply connected transitive $\fG_{\lambda}$-set. Then $\fZ=\fB(X)/\Gamma$ as above. Since $\fZ$ is simply connected, $\Gamma$ must be trivial, and so the proof is complete.
\end{proof}

\begin{lemma}
Let $f \colon Y \to X$ be a surjection in $\cE$. Suppose there is a surjection $X' \to X$ such that the base change $f' \colon Y' \to X'$ of $f$ is an isomorphism. Then $f$ is an isomorphism.
\end{lemma}

\begin{proof}
Consider the diagram
\begin{displaymath}
\xymatrix@C=3em{
Y' \times_{X'} Y' \ar@<2pt>[r] \ar@<-2pt>[r] \ar[d] & Y' \ar[d] \\
Y \times_X Y\ar@<2pt>[r] \ar@<-2pt>[r] & Y}
\end{displaymath}
Since $f'$ is an isomorphism, the two top horizontal arrows are equal. Since the vertical maps are surjective, it follows that the two bottom horizontal arrows are equal. Thus $f$ is injective, and therefore an isomorphism.
\end{proof}

\subsection{The main theorem}

The following theorem ties together the results of this section.

\begin{theorem}
The functor $\cE^s_{\lambda} \to \bS(\fG_{\lambda})$ given by $X \mapsto \fB(X)$ is fully faithful. If $\cE$ is exact then its image consists of those smooth $\fG_{\lambda}$-sets that are transitive and simply connected.
\end{theorem}

\begin{proof}
The functor is fully faithful by construction. It also follows by construction that $\fB(X)$ is smooth and transitive, and we have seen that it is simply connected (Proposition~\ref{prop:sc}). Finally, if $\cE$ is exact then we have seen that every simply connected transitive smooth $\fG_{\lambda}$-set is of the form $\fB(X)$  (Proposition~\ref{prop:sc} again).
\end{proof}

\begin{remark}
An interesting problem is to characterize the pro-oligomorphic groups arising from a finitely-powered regular category as in this section. We observe one simple obstruction here. If $X$ and $Y$ are objects of $\cE$ of type $\lambda$ then it is not hard to see that $\fB(X \times Y)$ appears with multiplicity one in $\fB(X) \times \fB(Y)$, meaning no other orbit is isomorphic to it. This observation extends to quotients by finite groups too. We thus see that a product of two transitive smooth $\fG_{\lambda}$-sets always contains a multiplicity one orbit. The oligomorphic group $\Aut(\bR, <)$ connected to the Delannoy category \cite{line} does not satisfy this property, and thus cannot arise from any $\cE$.
\end{remark}

\section{The second embedding theorem} \label{s:embed2}

The purpose of this section is to prove a version of Theorem~\ref{mainthm1}(a), which connects a finitely-powered regular category to a something like an oligomorphic monoid.

\subsection{The main objects}

Fix a finitely powered regular category $\cE$ throughout \S \ref{s:embed2}. We assume for simplicity that $\cE$ is exact and has countably many isomorphism classes; see Remark~\ref{rmk:hypo} for further discussion. Let $\Lambda$, $\cE^s_{\lambda}$, $\Omega_{\lambda}$ and $\fG_{\lambda}$ be as in the previous section. We assume that the pro-objects $\Omega_{\lambda}$ are indexed by $\bN$, which is possible by the countability hypothesis on $\cE$.

Define a category $\fM$ as follows. The objects are the elements of $\Lambda$, and
\begin{displaymath}
\Hom_{\fM}(\lambda,\mu)=\Hom_{\Pro(\cE)}(\Omega_{\mu}, \Omega_{\lambda}).
\end{displaymath}
Thus $\fM$ is just the opposite of the full subcategory of $\Pro(\cE)$ spanned by the $\Omega_{\lambda}$'s. We let $\fM^{\circ}$ be the maximal groupoid in $\fM$. In this category, we have
\begin{displaymath}
\Hom_{\fM^{\circ}}(\lambda, \mu) = \begin{cases}
\fG_{\lambda} & \text{if $\lambda=\mu$} \\
\emptyset & \text{otherwise.}
\end{cases}
\end{displaymath}
If $X$ is an object of $\cE$ of type $\lambda$, we regard $\fB(X)$ as a $\fM^{\circ}$-set supported at the object $\lambda$. We also define an $\fM$-set $\fA(X)$ by
\begin{displaymath}
\fA(X)_{\lambda} = \Hom_{\Pro(\cE)}(\Omega_{\lambda}, X).
\end{displaymath}
We note that $X \mapsto \fA(X)$ is functorial for all morphisms in $\cE$. We now establish a few simple properties of this construction.

\begin{proposition} \label{prop:resA}
The restriction of $\fA(X)$ to $\fM^{\circ}$ naturally decomposes as $\coprod_{Y \subset X} \fB(Y)$.
\end{proposition}

\begin{proof}
We have
\begin{displaymath}
\Hom_{\Pro(\cE)}(\Omega_{\lambda}, X) = \coprod_{Y \subset X} \Surj_{\Pro(\cE)}(\Omega_{\lambda}, Y),
\end{displaymath}
and so the result follows.
\end{proof}

\begin{proposition}
Consider morphisms $X \to Z$ and $Y \to Z$ in $\cE$. Then
\begin{displaymath}
\fA(X) \times_{\fA(Z)} \fA(Y) = \fA(X \times_Z Y).
\end{displaymath}
\end{proposition}

\begin{proof}
This is simply the mapping property for fiber products.
\end{proof}

\begin{proposition} \label{prop:A-faithful}
The functor $X \mapsto \fA(X)$ is faithful.
\end{proposition}

\begin{proof}
Let $f,g \colon X \rightrightarrows Y$ be given, and suppose $\fA(f)=\fA(g)$. Choose a surjection $q \colon \Omega_{\lambda} \to X$ for some $\lambda$. Then, by assumption, $f^*(q)=g^*(q)$. Since $q$ is an epimorphism, it follows that $f=g$, as required.
\end{proof}

We give $\fM$ the non-archimedean structure generated by the sets $\fA(X)$ for $X \in \cE$. By Proposition~\ref{prop:resA}, the induced non-archimedean structure on $\fM^{\circ}$ is the expected one, i.e., it is generated by the $\fB(X)$'s.

\begin{proposition}
$\fM$ is a pro-oligomorphic category.
\end{proposition}

\begin{proof}
We first show that the non-archimedean structure on $\fM$ is Hausdorff. Let $f,g \in \Hom_{\fM}(\lambda, \mu)$ be distinct. From how pro-objects work, there is a surjection $q \colon \Omega_{\lambda} \to X$ such that $q \circ f \ne q \circ g$. We thus see that $f^*,g^* \colon \fA(X)_{\lambda} \to \fA(X)_{\mu}$ act differently on $q$, as required. We already know that the groups $\fG_{\lambda}$ are pro-oligomorphic. Proposition~\ref{prop:resA} shows that a finitary smooth $\fM$-set restricts to a finitary $\fM^{\circ}$-set, which completes the proof.
\end{proof}

\begin{remark}
We note that the $\fA$ and $\fB$ sets, and the relationship between them, was observed in \cite[\S 3]{cantor} in the case where $\cE$ is the category of finite sets; the $Y(n)$ sets there correspond to our $\fA$'s, while then $X(n)$'s correspond to our $\fB$'s.
\end{remark}

\subsection{Classification of subobjects}

Let $X$ be an object of $\cE$. By Proposition~\ref{prop:resA}, we see that the $\fM^{\circ}$-subsets of $\fA(X)$ correspond to arbitrary subsets $\fc$ of $\Sub(X)$. We write $\fA(\fc)$ for the $\fM^{\circ}$-subset corresponding to $\fc$. We can characterize this correspondence as follows. First, $\fA(\fc)_{\lambda}$ consists of all maps $f \colon \Omega_{\lambda} \to X$ whose image belongs to $\fc$. And second, if $\fZ$ is a $\fM^{\circ}$-subset of $\fA(X)$ then $\fZ=\fA(\fc)$, where $\fc$ consists of all subobjects of $X$ that occur as the image of some morphism in $\fZ$. We say that $\fc$ is an \defn{ideal} of $\Sub(X)$ if it is downwards closed, i.e., $Y \in \fc$ and $Y' \subset Y$ implies $Y' \in \fc$.

\begin{proposition} \label{prop:A-sub}
$\fA(\fc)$ is a $\fM$-subset of $\fA(X)$ if and only if $\fc$ is an ideal of $\Sub(X)$.
\end{proposition}

\begin{proof}
Suppose $\fc$ is an ideal of $\Sub(X)$. We show that $\fA(\fc)$ is a $\fM$-subset of $\fA(X)$. Let $f \colon \Omega_{\lambda} \to X$ be given that belongs to $\fA(\fc)$, meaning the image of $f$ is in $\fc$. Given a morphism $g \colon \Omega_{\mu} \to \Omega_{\lambda}$, the image of $g^*(f)$ is contained in the image of $f$, and therefore belongs to $\fc$, and so $g^*(f)$ belongs to $\fA(\fc)$ as required.

Now suppose that $\fA(\fc)$ is a $\fM$-subset of $\fA(X)$. We show that $\fc$ is an ideal of $\Sub(X)$. Let $Y \in \fc$ and $Y' \subset Y$ be given. We can find a commutative diagram
\begin{displaymath}
\xymatrix@C=3em{
\Omega_{\mu} \ar[r]^h \ar[d]_{g'} & \Omega_{\lambda} \ar[d]^g \ar[rd]^f \\
Y' \ar[r]^j & Y \ar[r]^i & X}
\end{displaymath}
where $g$ and $g'$ are surjective, and $i$ and $j$ are the given inclusions. Indeed, first choose $g$ and $g'$ arbitrarily, and define $f=i \circ g$. Let $\Omega'=\Omega_{\lambda} \times_Y Y'$. Since $g$ is surjective, so too is the natural map $\Omega' \to Y'$, and so it is a morphism in $\cE^s_{\mu}$. We can thus lift $g'$ through this surjection, which gives the map $h$. (Here we are using the fact that $\Omega_{\mu}$ is a universal homogeneous pro-object in $\cE^s_{\mu}$ indexed by $\bN$, and thus maps to any pro-object in $\cE^s_{\mu}$ indexed by $\bN$; see \cite[Proposition~A.7]{homoten}.) Since $f$ belongs to $\fA(\fc)$, so does $h^*(f)$, and so $Y'$ belongs to $\fc$, as required.
\end{proof}

\begin{corollary} \label{cor:Bgen}
The $\fM$-set $\fA(X)$ is generated by $\fB(X) \subset \fA(X)$.
\end{corollary}

\begin{proof}
Indeed, the only ideal of $\Sub(X)$ containing $X$ is $\Sub(X)$ itself.
\end{proof}

\begin{corollary} \label{cor:Aquot}
Every smooth $\fM$-set is a quotient of a disjoint union of $\fA(X)$'s.
\end{corollary}

\begin{proof}
It suffices to show that every $\fM$-subset of $\fA(Y)$ is of the stated form. This follows from the previous proposition: indeed, we have a natural surjection $\coprod_{Y \in \fc} \fA(Y) \to \fA(\fc)$.
\end{proof}

\subsection{Monogenic sets}

We say that a smooth $\fM$-set is \defn{monogenic} if it is generated by a single element, or equivalently, by a single $\fM^{\circ}$-orbit.

\begin{proposition}
Let $\fZ$ be a monogenic smooth $\fM$-set. Then there is an object $X$ of $\cE$, a subgroup $\Gamma$ of $\Aut(X)$, and a surjection of $\fM$-sets $\fA(X)/\Gamma \to \fZ$ that is injective on $\fB(X)/\Gamma$.
\end{proposition}

\begin{proof}
Since $\fZ$ is monogenic, it is a quotient of some $\fA(X)$ (Corollary~\ref{cor:Aquot}). Choose a surjection $f \colon \fA(X) \to \fZ$ with $X$ minimal, in the sense that $f$ does not factor through $\fA(Y)$ for any proper quotient $Y$ of $X$; this is possible since $X$ has finitely many quotients by \S \ref{ss:fp}(b). Let $\fR=\fA(X) \times_{\fZ} \fA(X)$ be the equivalence relation corresponding to this quotient, and let $\fR_0=\fB(X) \times_{\fZ} \fB(X)$.

We claim that $f \colon \fB(X) \to \fZ$ is finite-to-one. Indeed, suppose this is not the case. Then, by the classification of transitive smooth $\fM^{\circ}$-sets (Proposition~\ref{prop:Gclass}), there is a proper quotient $X \to Y$ such that $f \vert_{\fB(X)}$ factors through $\fB(Y)$. Thus $\fR_0$ contains $\fB(X) \times_{\fB(Y)} \fB(X)$, and, in particular, $\fB(X \times_Y X)$. It follows that $\fR$ contains $\fB(X \times_Y X)$ as well, and thus the $\fM$-set generated by this orbit, which is $\fA(X \times_Y X)=\fA(X) \times_{\fA(Y)} \fA(X)$ by Proposition~\ref{prop:A-sub}. We thus see that $f$ factors through $\fA(Y)$, a contradiction. This proves the claim.

Now, by again appealing to the classification of transitive smooth $\fM^{\circ}$-sets, we see that the image $\fB(X)$ under $f$ has the form $\fB(X)/\Gamma$ for some subgroup $\Gamma \subset \Aut(X)$. For $\sigma \in \Gamma$, the maps $f$ and $f\sigma$ agree on $\fB(X)$, and thus on all of $\fA(X)$ since $\fB(X)$ generates $\fA(X)$ as an $\fM$-set (Corollary~\ref{cor:Bgen}). We thus see that $f$ factors through $\fA(X)/\Gamma$, and the induced map $\fA(X)/\Gamma \to \fZ$ is injective on $\fB(X)/\Gamma$. This completes the proof.
\end{proof}

\begin{proposition}
Let $\fZ$ be a monogenic smooth $\fM$-set. Then there is a unique $\fM^{\circ}$-orbit on $\fZ$ that generates $\fZ$ as an $\fM$-set.
\end{proposition}

\begin{proof}
Let $\fA(X)/\Gamma \to \fZ$ be as in the above proposition. Let $\fA_0(X) \subset \fA(X)$ be the union of the $\fB(Y)$ with $Y$ a proper subobject of $X$; this is an $\fM$-subset of $\fA(X)$. Let $\fZ_0$ be the image of $\fA_0(X)$, and let $\fZ' \cong \fB(X)/\Gamma$ be the image of $\fB(X)$. Then $\fZ=\fZ' \cup \fZ_0$, and $\fZ'$ generates $\fZ$ as an $\fM$-set, and $\fZ_0$ is an $\fM$-subobject of $\fZ$. To prove the proposition, it suffices to show that $\fZ_0$ does not contain $\fZ'$, for then $\fZ'$ will be the only $\fM^{\circ}$-orbit that generates.

To this end, it suffices to show that if $Y$ is a subobject of $X$ and $f \colon \fB(Y) \to \fB(X)/\Gamma$ is a map of $\fM^{\circ}$-sets, then $Y=X$. First note that $f$ lifts to a map $f' \colon \fB(Y) \to \fB(X)$ of $\fM^{\circ}$-sets; indeed, the pull-back of $\fB(X) \to \fB(X)/\Gamma$ along $f$ is a finite cover of $\fB(Y)$, and thus splits since $\fB(Y)$ is simply connected. The map $f'$ is induced by a surjection $Y \to X$ in $\cE$. Since $Y$ is a subobject of $X$, we have $\# \Sub(Y) \le \# \Sub(X)$. Since $X$ is a quotient of $Y$, we have the reverse inequality. Thus the inclusion $Y \to X$ induces an isomorphism $\Sub(Y)=\Sub(X)$, and so $Y=X$.
\end{proof}

Let $\fZ$ be a monogenic smooth $\fM$-set. We define the \defn{generic locus} in $\fZ$, denoted $\fZ^{\gen}$, to be the unique $\fM^{\circ}$-orbit that generates $\fZ$. If $f \colon \fZ \to \fZ'$ is a surjection of monogenic smooth $\fM$-sets then $f$ induces a map $f^{\gen} \colon \fZ^{\gen} \to (\fZ')^{\gen}$. We say that $f$ is \defn{generically finite} (resp.\ \defn{generically an isomorphism}) if $f^{\gen}$ is a finite cover (resp.\ an isomorphism). The above results imply:

\begin{corollary}
Let $\fZ$ be a monogenic smooth $\fM$-set. Then there is a generically finite surjection $\fA(X) \to \fZ$ for some $X$, which induces a generic isomorphism $\fA(X)/\Gamma \to \fZ$ for some subgroup $\Gamma \subset \Aut(X)$. The object $X$ is unique up to isomorphism.
\end{corollary}

\begin{proof}
We have explained everything except the uniquenes of $X$. This from the fact that if $\fB(X)/\Gamma$ is isomorphic to $\fB(X')/\Gamma'$ then $X \cong X'$; this statement, in turn, follows from the classification of transitive smooth $\fM^{\circ}$-sets.
\end{proof}

We say a monogenic smooth $\fM$-set $\fZ$ is \defn{simply connected} if any generically finite surjection $\fZ' \to \fZ$ of monogenic smooth $\fM$-sets is an isomorphism.

\begin{proposition} \label{prop:Asc}
The simply connected monogenic smooth $\fM$-sets are exactly the $\fA(X)$.
\end{proposition}

\begin{proof}
We first show that $\fA(X)$ is simply connected. First suppose that $\fA(Y) \to \fA(X)$ is a generically finite surjection. This induces a finite cover $\fB(Y) \to \fB(X)$, and so $Y \to X$ is an isomorphism (Proposition~\ref{prop:sc}); thus $\fA(Y) \to \fA(X)$ is an isomorphism too. Now consider an arbitrary generically finite surjection $\fZ' \to \fA(X)$ with $\fZ'$ monogenic and smooth. Let $\fA(Y) \to \fZ'$ be a generically finite surjection. By the previous case, the composite map $\fA(Y) \to \fA(X)$ is an isomorphism, and so $\fZ' \to \fA(X)$ is an isomorphism. Thus $\fA(X)$ is simply connected.

Now suppose that $\fZ$ is a simply connected monogenic smooth $\fM$-set. There is then a generically finite surjection $\fA(X) \to \fZ$, which is an isomorphism since $\fZ$ is simply connected. This completes the proof.
\end{proof}

\subsection{Induction}

Any $\fM$-set can be restricted to an $\fM^{\circ}$-set. This operation preserves smoothness, and so we have a restriction functor
\begin{displaymath}
\Res \colon \hat{\bS}(\fM) \to \hat{\bS}(\fM^{\circ})
\end{displaymath}
Since the source and target are Grothendieck topoi and $\Res$ is continuous, it admits a left adjoint $\Ind$, called \defn{induction}.

\begin{proposition} \label{prop:Ind-B}
We have $\Ind(\fB(X)) = \fA(X)$.
\end{proposition}

\begin{proof}
Let $\fZ$ be a smooth $\fM$-set. We must show that the map
\begin{displaymath}
\Hom_{\fM}(\fA(X), \fZ) \to \Hom_{\fM^{\circ}}(\fB(X), \fZ),
\end{displaymath}
induced by restricting from $\fA(X)$ to $\fB(X)$, is an isomorphism. This map is injective since $\fB(X)$ generates $\fA(X)$ (Corollary~\ref{cor:Bgen}). We now show it is surjective. Thus let a $\fM^{\circ}$-map $f \colon \fB(X) \to \fZ$ be given. Let $\fZ'$ be the $\fM$-subset of $\fZ$ generated by the image of $f$, which is monogenic. Let $\fA(Y)/\Gamma \to \fZ'$ be a generic isomorphism. The surjection $\fB(X) \to (\fZ')^{\gen}=\fB(Y)/\Gamma$ is induced from a surjection $X \to Y$, which in turn induces a surjection $\fA(X) \to \fA(Y)$. The composite $\fA(X) \to \fZ$ extends the original map $f$, which completes the proof.
\end{proof}

\begin{corollary}
Induction preserves finitary sets.
\end{corollary}

\begin{proof}
Since $\Ind$ is co-continuous, we have $\Ind(\fB(X)/\Gamma)=\fA(X)/\Gamma$ for a subgroup $\Gamma \subset \Aut(X)$. Since every finitary smooth $\fM^{\circ}$-set is a finite union of sets of the form $\fB(X)/\Gamma$ (Proposition~\ref{prop:Gclass}), the result follows.
\end{proof}

\subsection{The main theorem}

We are now ready to prove our main result relating $\cE$ and $\fM$:

\begin{theorem}
The functor $\cE \to \bS(\fM)$ given by $X \mapsto \fA(X)$ is fully faithful. Its essential image consists of those smooth $\fM$-sets that are monogenic and simply connected.
\end{theorem}

\begin{proof}
We first prove that the functor is fully faithful. The natural map
\begin{displaymath}
\Hom_{\cE}(X,Y) \to \Hom_{\fM}(\fA(X), \fA(Y))
\end{displaymath}
is injective by Proposition~\ref{prop:A-faithful}. It thus suffices to show the two sides have the same cardinality. We have natural isomorphisms
\begin{align*}
\Hom_{\fM}(\fA(X), \fA(Y))
&= \Hom_{\fM^{\circ}}(\fB(X), \fA(Y))
= \coprod_{Y' \subset Y} \Hom_{\fM^{\circ}}(\fB(X), \fB(Y')) \\
&= \coprod_{Y' \subset Y} \Hom_{\cE^s}(X, Y')
= \Hom_{\cE}(X, Y),
\end{align*}
as required. In the first step we used the calculation of induction (Proposition~\ref{prop:Ind-B}), and in the second step the calculation of restriction (Proposition~\ref{prop:resA}). The characterization of the essential image follows from Proposition~\ref{prop:Asc}.
\end{proof}

\begin{remark} \label{rmk:hypo}
We have assumed in this section that $\cE$ has countably many isomorphism classes. This covers all cases of interest to us at the present time. This assumption was used in the proof of Proposition~\ref{prop:A-sub} to lift a map from $\Omega_{\mu}$. To remove this hypothesis, one can either carefully keep track of the sizes of index sets of pro-objects, or use an ultraproduct trick as in \cite[\S 6.5]{pregalois}. We also assumed that $\cE$ is exact. This is only needed to obtain a complete classification of the monogenic simply connected smooth $\fM$-sets, as one can see by passing to the exact completion (\S \ref{ss:exact}).
\end{remark}

\subsection{Sheaves}

We have just seen that a certain small piece of the category $\hat{\bS}(\fM)$ of smooth $\fM$-sets is equivalent to $\cE$. It is natural to wonder if there is a simple way of recovering the entire category from $\cE$. We now show that this is indeed possible.

Define a Grothendieck topology on $\cE$ by taking the covering families to be surjective morphisms. We denote the resulting site by $\cE_{\rm re}$, and write $\Sh(\cE_{\rm re})$ for the category of sheaves on $\cE_{\rm re}$. Explicitly, such a sheaf is a functor $F \colon \cE^{\op} \to \Set$ such that for any surjection $Y \to X$, the sequence
\begin{displaymath}
F(X) \to F(Y) \rightrightarrows F(Y \times_X Y)
\end{displaymath}
is an equilizer diagram. If $X$ is an object of $\cE$ then the representable presheaf $h_X=\Hom_{\cE}(-, X)$ is a sheaf on $\cE_{\rm re}$. This notion of sheaf is commonly used in connection to regular categories; see, e.g., \cite[\S 4]{Barr} or \cite[\S 2]{Lack} for details.

The following is the main result we are after in this direction.

\begin{theorem}
We have a natural equivalence of categories $\Sh(\cE_{\mathrm{re}}) = \hat{\bS}(\fM)$.
\end{theorem}

\begin{proof}
Since this result is somewhat tangential to the main point of the paper, we omit some details in this proof. Let $F$ be a sheaf on $\cE_{\rm re}$. For $\lambda \in \Lambda$, put
\begin{displaymath}
\Phi_{\lambda}(F) = F(\Omega_{\lambda}) = \varinjlim F(\Omega_{\lambda,i}),
\end{displaymath}
where $\Omega_{\lambda,i}$ are the various terms in the pro-object $\Omega_{\lambda}$. The functor $\Phi_{\lambda}$ preserves finite limits and finite coproducts. Suppose $f \colon F \to F'$ is a surjection of sheaves. We claim that $\Phi_{\lambda}(f)$ is a surjection of sets. Indeed, given $x \in F'(\Omega_{\lambda,i})$, there is some surjection $T \to \Omega_{\lambda,i}$ such that $x$ lifts through $f$ to an element $F(T)$. Since $\Omega_{\lambda}$ is f-projective, the natural map $\Omega_{\lambda} \to \Omega_{\lambda,i}$ lifts through $T$. This shows that $x$ lifts to an element of $F(\Omega_{\lambda,j})$ for some $j$, i.e., to an element of $\Phi_{\lambda}(F)$. This proves the claim.

The $\Phi_{\lambda}(F)$'s, as $\lambda$ varies, define an $\fM$-set $\Phi(F)$. Clearly, we have $\Phi(h_X)=\fA(X)$. If $F$ is any sheaf then $F$ is a quotient of a disjoint union of representable sheaves, and so we see that $\Phi(F)$ is a quotient of a smooth $\fM$-set, and thus smooth itself. We thus have a functor
\begin{displaymath}
\Phi \colon \Sh(\cE_{\rm re}) \to \hat{\bS}(\fM).
\end{displaymath}
We aim to show that this functor is an equivalence.

Let $X$ be an object of $\cE$ of type $\lambda$, choose a surjection $q_X \colon \Omega_{\lambda} \to X$, and let $U(X) \subset \fG_{\lambda}$ be the stabilizer of $q_X$. If $F$ is a sheaf on $\cE_{\rm re}$ then the natural map
\begin{displaymath}
q_X^* \colon F(X) \to \Phi_{\lambda}(F)^{U(X)}
\end{displaymath}
is a bijection; intuitively, $q_X$ is a $U(X)$-torsor. It follows that $\Phi$ is faithful. We have
\begin{displaymath}
\Phi_{\lambda}(F)^{U(X)} = \Hom_{\fM^{\circ}}(\fB(X), \Phi(F)) = \Hom_{\fM}(\fA(X), \Phi(F)),
\end{displaymath}
where in the first step we used the identification $\fB(X)=\fG_{\lambda}/U(X)$, and in the second step we used the computation of induction (Proposition~\ref{prop:Ind-B}). Since $\fA(X)=\Phi(h_X)$, we see that the natural map
\begin{displaymath}
\Hom(h_X, F) \to \Hom_{\fM}(\Phi(h_X), \Phi(F))
\end{displaymath}
is a bijection. Now let $F'$ be an arbitrary sheaf. Choosing a presentation of $F'$ by representable sheaves and using the above bijection, we find that the natural map
\begin{displaymath}
\Hom(F', F) \to \Hom_{\fM}(\Phi(F'), \Phi(F))
\end{displaymath}
is a bijection. Thus $\Phi$ is fully faithful.

Finally, suppose that $\fZ$ is a given smooth $\fM$-set. Choose a presentation of $\fZ$ by sets of the form $\fA(X)$, which is possible by Corollary~\ref{cor:Aquot}. Define $F$ to be the sheaf obtained using the corresponding presentation of representable sheaves. We then have $\Phi(F) \cong \fZ$, which shows that $\Phi$ is essentially surjective.
\end{proof}

\section{Degree functions and measures}

The purpose of this section is to prove the first part of Theorem~\ref{mainthm2}, which relates degree functions (as defined by Knop) to measures (as defined by Harman--Snowden).

\subsection{Degree functions} \label{ss:degree}

Let $\cE$ be a finitely-powered regular category. Recall that $\{E_{\lambda}\}_{\lambda \in \Lambda}$ are the subobjects of $\bone$. Let $1 \in \Lambda$ be the index such that $E_1=\bone$. We say that an object $X$ of $\cE$ is \defn{principal} if it has type~1, meaning that the map $X \to \bone$ is surjective.

\begin{definition} \label{defn:degree}
A \defn{degree function} for $\cE$ valued in a ring $k$ is a rule $\nu$ that assigns to each surjection $f \colon Y \to X$ of principal objects a value $\nu(f) \in k$ such that the following hold:
\begin{enumerate}
\item $\nu(f)=1$ if $f$ is an isomorphism.
\item $\nu(g \circ f)=\nu(g) \cdot \nu(f)$ when defined.
\item If $f \colon Y \to X$ is a surjection of principal objects, $X' \to X$ is any map of principal objects, and $f' \colon Y' \to X'$ is the base change of $f$, then $\nu(f')=\nu(f)$. \qedhere
\end{enumerate}
\end{definition}

\begin{example}
The \defn{trivial degree function} assigns~1 to each surjection.
\end{example}

The notion of degree function does not refer to addition in the ring $k$, and so one can make sense of a degree function valued in an arbitrary commutative monoid\footnote{In fact, commutativity is not necessary, but we assume it for simplicity.}. There is a universal degree function valued in a commutative monoid $\Pi(\cE)$. To define this monoid, take the free commutative monoid generated by morphisms of principal objects and quotient by relations corresponding to (a)--(c). For a surjection $f$ of principal objects, we write $\{ f \}$ for the associated class in $\Pi(\cE)$. Giving a degree function valued in a commutative monoid $M$ is then equivalent to giving a monoid homomorphism $\Pi(\cE) \to M$. In particular, giving a degree function valued in a ring $k$ is equivalent to giving a ring homomorphism $\bZ[\Pi(\cE)] \to k$, where $\bZ[\Pi(\cE)]$ denotes the monoid algebra. The ring $\bZ[\Pi(\cE)]$ is what Knop called $\rK(\cE)$ \cite[\S 8]{Knop2}.

\begin{remark}
Knop \cite[Definition~3.1]{Knop2} defines degree functions using all objects, not just principal objects. This difference does not really change much, especially due to \cite[Theorem~3.6]{Knop2}.
\end{remark}

\subsection{Measures}

Let $\cC$ be an order in a pre-Galois category. The following is a key idea introduced in \cite{repst}:

\begin{definition} \label{defn:meas}
A \defn{measure} for $\cC$ valued in a ring $k$ is a rule $\mu$ that assigns to each morphism $f \colon Y \to X$ of atoms in $\cC$ a value $\mu(f) \in k$ such that the following conditions hold:
\begin{enumerate}
\item $\mu(f)=1$ if $f$ is an isomorphism.
\item $\mu(g \circ f)=\mu(g) \cdot \mu(f)$ when defined.
\item Suppose $X' \to X$ is a morphism of atoms, let $Y'=Y \times_X X'$, and let $f' \colon Y' \to X'$ be the natural map. Let $Y'=\bigsqcup_{i=1}^n Y'_i$ be the decomposition of $Y'$ into atoms, and let $f'_i$ be the restriction of $\phi'$ to $Y'_i$. Then $\mu(f)=\sum_{i=1}^n \mu(f'_i)$. \qedhere
\end{enumerate}
\end{definition}

Suppose $\mu$ is a measure. There are two pieces of notation we introduce. First, suppose $f \colon Y \to X$ is a map with $X$ an atom. Let $Y=\bigcup_{i=1}^n Y_i$ be the decomposition of $Y$ into atoms, and let $f_i$ be the restriction of $f$ to $Y_i$. We then define $\mu(f)=\sum_{i=1}^n \mu(f_i)$. With this convention, axiom~(c) above takes the form $\mu(f)=\mu(f')$. Second, for an object $X$ of $\cC$, we put $\mu(X)=\mu(X \to \bone)$.

We say that the measure $\mu$ is \defn{regular} if $\mu(X)$ is a unit of $k$ for all objects $X$. Suppose this is the case, and let $f \colon Y \to X$ be a morphism in $\cC$, with $X$ an atom. By considering the composition of $f$ with the map $X \to \bone$, we find that $\mu(f) \mu(X)=\mu(Y)$, and so $\mu(f)=\mu(Y) \mu(X)^{-1}$. We thus see that the value of $\mu$ on morphisms is entirely determined from its values on objects. See \cite[\S 3.5]{repst} for more details.

As with degree functions, there is a universal measure valued in a ring $\Theta(\cC)$: giving a measure valued in $k$ is the same as giving a ring homomorphism $\Theta(\cC) \to k$. See \cite[\S 4.5]{repst} for details. We let $[f]$ be the class of $f \colon Y \to X$ in $\Theta(\cC)$.

\subsection{Constant maps} \label{ss:constant}

Let $\mu$ be a measure for $\cC$, as above. Let $f \colon Y \to X$ be a map in $\cC$. Write $X=\bigsqcup_{i=1}^n X_i$, with $X_i$ an atom, let $Y_i=f^{-1}(X_i)$, and let $f_i \colon Y_i \to X_i$ be the induced map. We say that $f$ is \defn{$\mu$-constant} if $\mu(f_i)$ is independent of $i$; assuming $n \ne 0$, we then put $\mu(f)=\mu(f_i)$ for any $i$. We now establish two simple properties of $\mu$-constant maps.

\begin{proposition} \label{prop:mu-const-1}
Suppose $f \colon Y \to X$ is $\mu$-constant. Let $X' \to X$ be an arbitrary map, with $X'$ non-empty, and let $f' \colon Y' \to X'$ be the base change of $f$. Then $f'$ is also $\mu$-constant, and $\mu(f')=\mu(f)$.
\end{proposition}

\begin{proof}
Write $X'=\bigsqcup_{i=1} X'_i$, with $X'_i$ an atom, and let $f'_i$ be the restriction of $f'$ to the inverse image of $X'_i$. Let $X_i$ be the image of $X'_i$ in $X$, which is an atom, and let $f_i$ be the restriction of $f$ to the inverse image of $X_i$. Then $f'_i$ is the base change of $f_i$, and so $\mu(f'_i)=\mu(f_i)=\mu(f)$. The result follows.
\end{proof}

\begin{proposition} \label{prop:mu-const-2}
If $f \colon Z \to Y$ and $g \colon Y \to X$ are $\mu$-constant, with $X$ non-empty, then so is $g \circ f$, and $\mu(g \circ f)=\mu(g) \cdot \mu(f)$.
\end{proposition}

\begin{proof}
Write $X=\bigsqcup_{i=1}^n X_i$ with $X_i$ an atom. Let $Y_i=g^{-1}(Y)$ and $Z_i=f^{-1}(Y_i)$, and let $f_i$ and $g_i$ be the restrictions of $f$ and $g$ to $Z_i$ and $Y_i$. Write $Y_i=\bigsqcup_{j=1}^{m(i)} Y_{i,j}$, with $Y_{i,j}$ an atom, let $Z_{i,j}=g^{-1}(Y_{i,j})$, and let $f_{i,j}$ and $g_{i,j}$ be the restrictions of $f$ and $g$ to $Z_{i,j}$ and $Y_{i,j}$. We have
\begin{displaymath}
\mu(g_i \circ f_i) = \sum_{j=1}^{m(i)} \mu(g_{i,j} \circ f_{i,j}) = \sum_{j=1}^{m(i)} \mu(g_{i,j}) \mu(f_{i,j}) = \mu(f) \sum_{j=1}^{m(i)} \mu(g_{i,j}) = \mu(f) \mu(g_i) = \mu(f) \mu(g),
\end{displaymath}
and so the result follows.
\end{proof}

\subsection{Measures from degree functions} \label{ss:meas-compare}

Fix a finitely-powered regular category $\cE$. Let $\fG=\fG_1$ be the pro-oligomorphic group associated to $\cE^s_1$, let $\cC_0=\bS(\fG)$ be the category of finitary smooth $\fG$-sets, and let $\cC$ be the order in $\cC_0$ consisting those objects whose atoms have the form $\fB(X)$, with $X \in \cE^s_1$. (This is indeed an order by Proposition~ \ref{prop:B-prod}.)

\begin{construction} \label{con:main}
Let $\nu$ be a degree function for $\cE$ valued in a ring $k$. We define a measure $\mu$ for $\cC$ valued in $k$ as follows. Let $f \colon Y \to X$ be a surjection of principal objects in $\cE$. Put
\begin{displaymath}
\mu(f \colon \fB(Y) \to \fB(X)) = \sum_{Z \subset Y, f(Z)=X} \bmu(Z,Y) \cdot \nu(f \vert_Z),
\end{displaymath}
where $\bmu$ is the M\"obius function of the lattice $\Sub(Y)$. We note that
\begin{displaymath}
\nu(f) = \sum_{Z \subset Y, f(Z)=X} \mu(f \colon \fB(Z) \to \fB(X))
\end{displaymath}
by M\"obius inversion, and so one can recover $\nu$ from $\mu$.
\end{construction}

\begin{remark}
Construction~\ref{con:main} is essentially \cite[(8.5)]{Knop2}: our $\mu(f)$ is Knop's $\omega_e$.
\end{remark}

The following proposition shows that $\mu$ is indeed a measure.

\begin{proposition}
The above rule $\mu$ is a measure for $\cC$.
\end{proposition}

\begin{proof}
We must verify the three axioms of Definition~\ref{defn:meas}. Axiom~(a) is clear. Axiom~(b) was already proved by Knop \cite[Lemma~8.4]{Knop2}. We now handle~(c). Consider a cartesian diagram of principal objects in $\cE$
\begin{displaymath}
\xymatrix@C=3em{
Y' \ar[r]^{g'} \ar[d]_{f'} & Y \ar[d]^f \\
X' \ar[r]^g & X }
\end{displaymath}
where $f$ and $f'$ are surjective. Let $Y'_i$, for $1 \le i \le n$, be the ample subobjects of $Y'$, and let $f'_i \colon Y'_i \to X'$ be the restriction of $f'$. Then we have a cartesian diagram in $\cC$
\begin{displaymath}
\xymatrix@C=3em{
\coprod_{i=1}^n \fB(Y'_i) \ar[r]^-{g'} \ar[d] & \fB(Y) \ar[d]^f \\
\fB(X') \ar[r]^-g & \fB(X) }
\end{displaymath}
Note that every cartesian diagram in $\cC$ where the top right, bottom left, and bottom right objects are atoms has this form, and so it suffices to verify~(c) for such diagrams. In this setting, (c) takes the form
\begin{displaymath}
\mu(f \colon \fB(Y) \to \fB(X)) = \sum_{i=1}^n \mu(f'_i \colon \fB(Y'_i) \to \fB(X')).
\end{displaymath}
Call the left side above $\alpha$ and the right side $\beta$. By definition of $\mu$, we have
\begin{displaymath}
\alpha = \sum_{Z \subset Y, f(Z)=X} \bmu(Z,Y) \cdot \nu(f \vert_Z).
\end{displaymath}
Now, the pullback of $Z \to X$ along $X' \to X$ is the map $(g')^{-1}(Z) \to X'$. Since $\nu$ is invariant under base change, we thus find
\begin{displaymath}
\alpha = \sum_{Z \subset Y, f(Z)=X} \bmu(Z,Y) \cdot \nu(f' \vert_{(g')^{-1}(Z)})
\end{displaymath}
We now express $\nu$ in terms of $\mu$ using M\"obius inversion:
\begin{displaymath}
\alpha = \sum_{\substack{Z \subset Y\\ f(Z)=X}} \sum_{\substack{Z' \subset (g')^{-1}(Z) \\ f'(Z')=X'}} \bmu(Z,Y) \cdot \mu(f' \colon \fB(Z') \to \fB(X'))
\end{displaymath}
We now switch the order of summation. Thus we first let $Z'$ vary over all subobjects of $Y'$ that surject onto $X'$, and then let $Z$ vary over subobjects of $Y$ that contain $g'(Z')$; note that this ensures $Z$ surjects onto $X$. We thus have:
\begin{displaymath}
\alpha = \sum_{\substack{Z' \subset Y' \\ f'(Z')=X'}} \big( \sum_{g'(Z') \subset Z \subset Y} \bmu(Z,Y) \big) \cdot \mu(f' \colon \fB(Z') \to \fB(X'))
\end{displaymath}
The inner sum vanishses unless $f'(Z')=Y$, in which case it equals~1. We thus find
\begin{displaymath}
\alpha = \sum_{Z'} \mu(f' \colon \fB(Z') \to \fB(X')),
\end{displaymath}
where the sum is over subobjects $Z'$ of $Y$ such that $f'(Z')=X'$ and $g'(Z')=Y$. We are thus exactly summing over the ample subobjects of $Y'$, and so the above sum is equal to $\beta$, as required.
\end{proof}

\begin{corollary}
There is a ring homomorphism
\begin{displaymath}
\phi \colon \Theta(\cC) \to \bZ[\Pi(\cE)]
\end{displaymath}
uniquely characterized by the following property: if $f \colon Y \to X$ is a surjection of principal objects in $\cE$ then
\begin{displaymath}
\phi([f \colon \fB(Y) \to \fB(X)]) = \sum_{Z \subset Y, f(Z)=X} \bmu(Z,Y) \cdot \{ f \vert_Z \}.
\end{displaymath}
\end{corollary}

\begin{proof}
Apply the Construction~\ref{con:main} to the universal degree function.
\end{proof}

\subsection{Characterization}

Maintain the above set-up. We have just seen that every degree function gives rise to a measure. We now investigate which measures are thus obtained. For a principal object $X$ of $\cE$, define $\fA_1(X)$ to be the union of $\fB(Y)$ over the principal subobjects $Y$ of $X$. This is an object of $\cC$, and naturally a $\fM^{\circ}$-subset of $\fA(X)$.

\begin{proposition} \label{prop:char}
Let $\mu$ be a measure for $\cC$. Then the following are equivalent:
\begin{enumerate}
\item The measure $\mu$ arises from a degree function via Construction~\ref{con:main}.
\item For every surjection $Y \to X$ of principal objects, the induced map $\fA_1(Y) \to \fA_1(X)$ is $\mu$-constant (see \S \ref{ss:constant}).
\end{enumerate}
\end{proposition}

\begin{proof}
Assume (b) holds. Define $\nu(f) = \mu(\fA_1(Y) \to \fA_1(X))$. It follows from Propositions~\ref{prop:mu-const-1} and~\ref{prop:mu-const-2} that $\nu$ is a degree function. By computing the measure of $\fA_1(Y) \to \fA_1(X)$ after taking the base change along $\fB(X) \to \fA_1(X)$, we find
\begin{displaymath}
\nu(f) = \sum_{Y' \subset Y, f(Y')=X} \mu(f \colon \fB(Y') \to \fB(X)),
\end{displaymath}
which shows that $\mu$ comes from $\nu$ via the above construction.

Now suppose that $\mu$ comes from the degree function $\nu$. Let $f \colon Y \to X$ be a surjection of principal objects. Let $X_1, \ldots, X_n$ be the principal subobjects of $X$, and let $\fY_i$ be the inverse image of $X_i$ in $\fA_1(Y)$ under the map induced by $f$. We claim that
\begin{displaymath}
\mu(\fY_i \to \fB(X_i)) = \nu(f)
\end{displaymath}
for all $i$, which will show that (b) holds. Let $Y_i$ be the inverse image of $X_i$ in $Y$, and let $f_i$ be the restriction of $f$ to $Y_i$. By definition of degree function, we have $\nu(f)=\nu(f_i)$. Since $\mu$ is associated to $\nu$, we have
\begin{displaymath}
\nu(f_i) = \sum_{Z \subset Y_i, f(Z)=X_i} \mu(\fB(Z) \to \fB(X)).
\end{displaymath}
But the right side is exactly $\mu(\fY_i \to \fB(X_i))$, and so the result follows.
\end{proof}

\section{Tensor categories}

The purpose of this section is to prove the second part of Theorem~\ref{mainthm2}, which relates the tensor categories defined by Knop to those defined by Harman--Snowden.

\subsection{Knop's category}

Let $\cE$ be a finitely-powered regular category and let $\nu$ be a degree function on $\cE$ valued in a ring $k$. Recall that an object $X$ of $\cE$ is \defn{principal} if $X \to \bone$ is surjective. We now recall the construction of Knop's tensor category. For the purposes of this paper, a \defn{tensor category} is a $k$-linear category equipped with a symmetric monoidal structure that is $k$-bilinear.

\begin{construction}
We define a $k$-linear tensor category $\TT^0(\cE; \nu)$ as follows.
\begin{itemize}
\item Objects are formal symbols $[X]$ where $X$ is a principal object of $\cE$.
\item Morphisms $[X] \to [Y]$ are formal $k$-linear combinations of symbols $[A]$, where $A$ is a principal relation from $X$ to $Y$.
\item Composition of morphisms $[A] \colon [X] \to [Y]$ and $[B] \colon [Y] \to [Z]$ is defined as follows. If $B \circ A$ is not principal then $[B] \circ [A]=0$. Otherwise, $[B] \circ [A] = \nu(f) \cdot [B \circ A]$, where $f \colon B \times_X A \to B \circ A$ is the natural surjection.
\item The tensor product on objects is given by $[X] \otimes [Y] = [X \times Y]$. On morphisms it is described as follows. Let $[A] \colon [X] \to [Y]$ and $[B] \colon [X'] \to [Y']$ be given. Then $[A] \otimes [B]=[A \times B]$, where $A \times B$ is regarded as a relation from $X \times X'$ to $Y \times Y'$ in the obvious manner.
\end{itemize}
Finally, we define $\TT(\cE; \nu)$ to be the additive--Karoubi\footnote{Meaning one first takes the additive envelope, and then the Karoubi envelope. See \cite[\S 2.3]{dblexp} for details.} envelope of $\TT^0(\cE; \nu)$
\end{construction}

\subsection{The Harman--Snowden category}

Let $\cC$ be an order in a pre-Galois category and let $\mu$ be a measure for $\cC$ valued in a ring $k$. For an objet $X$ of $\cC$, let $X^{\orb}$ be the ``orbit space'' of $X$, i.e., the set of atomic subobjects of $X$, and let $\cF(X)$ be the space of all functions $X^{\orb} \to k$. Note that if we identify $\cC$ with a category of $\fG$-sets, for a pro-oligomorphic group $\fG$, then $\cF(X)$ is the space of $\fG$-invariant $k$-valued functions on $X$. If $Y$ is a subobject of $X$, we let $1_Y \in \cF(X)$ be the function that is~1 on the atoms in $Y$, and~0 on the atoms not in $Y$. The functions $1_Y$, with $Y$ an atomic subobject of $X$, form a $k$-basis of $\cF(X)$.

Suppose $f \colon Y \to X$ is a morphism in $\cC$. We define a push-forward map
\begin{displaymath}
f_* \colon \cF(Y) \to \cF(X)
\end{displaymath}
by setting $f_*(1_A)=c \cdot 1_{f(A)}$ where $A$ is an atomic subobjet of $Y$ and $c=\mu(A \to f(A))$. This is equivalent to the push-forward operation defined in \cite[\S 3.4]{repst}. In particular, this operation satisfies many of the properties one would expect, e.g., $(fg)_*=f_* g_*$ when defined.

Let $X$ and $Y$ be objects of $\cC$. We define a \defn{$Y \times X$ matrix} to be an element of $\cF(Y \times X)$, and we write $\Mat_{Y,X}$ for the space of such matrices. If $A$ is a $Y \times X$ matrix and $B$ is a $Z \times Y$ matrix, then we define $BA$ to be the $Z \times X$ matrix $(p_{13})_*(p_{12}^*(A) p_{23}^*(B))$, where $p_{ij}$ is the projection map on $Z \times Y \times X$. The definitions here agree with those in \cite[\S 7]{repst}, with the caveat that here we are only considering $\fG$-invariant matrices. In particular, we see that matrix multiplication is associated (when defined) and the identity matrix (i.e., $1_{\Delta(X)} \in \Mat_{X,X}$) is the identity for matrix multiplication.

We are now ready to introduce the Harman--Snowden tensor category:

\begin{construction}
We define a $k$-linear tenor category $\uPerm(\cC; \mu)$ as follows:
\begin{itemize}
\item Objects are formal symbols $\Vec_X$, where $X$ is an object of $\cC$.
\item Morphisms $\Vec_X \to \Vec_Y$ are $Y \times X$ matrices.
\item Composition is given by matrix multiplication.
\item The tensor product is given by $\Vec_X \otimes \Vec_Y=\Vec_{X \times Y}$ on objects, and by the obvious analog of the Kronecker product on morphisms. \qedhere
\end{itemize}
\end{construction}

We refer to \cite[\S 8]{repst} for additional details. We note that the category $\uPerm(\cC; \mu)$ is additive. On objects, the direct sum is given by
\begin{displaymath}
\Vec_X \oplus \Vec_Y = \Vec_{X \amalg Y}.
\end{displaymath}
The category $\uPerm(\cC; \mu)$ is not Karoubian in general, though.

\subsection{Comparison}

Maintain the set-up from \S \ref{ss:meas-compare}. Fix a degree function $\nu$ on $\cE$, and let $\mu$ be the measure on $\cC$ associated to $\nu$ via Construction~\ref{con:main}. Recall that for an object $X$ of $\cE$, we write $\fA_1(X)$ for the union of the $\fB(Y)$ as $Y$ varies over principal subobjects of $X$. Observe that $\fA_1(X) \times \fA_1(Y) = \fA_1(X \times Y)$. For a subobject $A \subset X$, we write $1_A$ for the function $1_{\fA_1(A)}$ in $\cF(\fA_1(X))$. We require the following key lemma.

\begin{lemma}
Let $f \colon Y \to X$ be a morphism of principal objects. Let $A$ be a principal subobject of $Y$, and let $B=f(A)$ be its image. Then
\begin{displaymath}
f_*(1_A) = \nu(f \colon A \to B) \cdot 1_B.
\end{displaymath}
\end{lemma}

\begin{proof}
By the proof of Proposition~\ref{prop:char}, the map $\fA_1(f) \colon \fA_1(A) \to \fA_1(B)$ is $\mu$-constant, and $\mu(\fA_1(f))=\nu(f)$. The result thus follows.
\end{proof}

\begin{remark}
When $\cE$ is the category of finite sets, the above lemma was proved in \cite[Lemma~5.4]{cantor} and \cite[Lemma~6.3]{cantor} (the two measures on $\cC$ were treated separately).
\end{remark}

We now arrive at our main theorem.

\begin{theorem} \label{thm:tensor}
There is an equivalence of tensor categories
\begin{displaymath}
\Phi \colon \TT(\cE; \nu) \to \uPerm(\cC; \mu)^{\rm kar}
\end{displaymath}
satisfying $\Phi([X])=\Vec_{\fA_1(X)}$.
\end{theorem}

\begin{proof}
We first define a functor $\Phi^0 \colon \TT^0(\cE; \nu) \to \uPerm(\cC; \mu)$. On objects, we put $\Phi^0([X])=\Vec_{\fA_1(X)}$. For a morphism $[A] \colon [X] \to [Y]$, we put $\Phi^0([A]) = 1_A$, where $1_A \in \cF(\fA_1(Y \times X))$ is regarded as a $\fA_1(Y) \times \fA_1(X)$ matrix. Suppose $[B] \colon [Y] \to [Z]$ is a second morphism. Let $D=B \times_Y A$, which is a subobject of $Z \times Y \times X$, and let $C$ be its image in $Z \times X$. The product of the matrices $1_A$ and $1_B$ is the push-forward of $1_D$ to $Z \times X$. By the previous lemma, this is $\nu(D \to C) 1_C$; note that this vanishes if $C$ is not principal. On the other hand, the composition $[B] \circ [A]$ is equal to $\nu(D \to C) [C]$ if $C$ is principal, and~0 otherwise. This shows that $\Phi^0$ is compatible with composition. It is clearly compatible with identity maps, and is thus a functor. It is fully faithful, as the $1_A$ form a basis for $\fF(\fA_1(Y \times X))$ as $A$ varies over principal subobjects of $Y \times X$.

We have natural identifications
\begin{align*}
\Phi^0([X] \otimes [Y]) &= \Phi^0([X \times Y]) = \Vec_{\fA_1(X \times Y)} = \Vec_{\fA_1(X) \times \fA_1(Y)} \\
&= \Vec_{\fA_1(X)} \otimes \Vec_{\fA_1(Y)} = \Phi^0([X]) \otimes \Phi^0([Y]),
\end{align*}
and so $\Phi^0$ is compatible with tensor products on objects. The above identifications are natural; we omit the proof. Thus $\Phi^0$ is a tensor functor.

Taking additive--Karoubi envelopes, $\Phi^0$ induces a functor $\Phi$ as in the statement of the theorem; note that $\uPerm(\cC; \mu)$ is additive, and thus its own additive envelope. The functor $\Phi$ is still fully faithful and a tensor functor. It is also essentially surjective. Indeed, if $X$ is a principal object of $\cE$ then $\Vec_{\fA_1(X)}$ is in the image of $\Phi^0$, and so its summand $\Vec_{\fB(X)}$ is in the essential image of $\Phi$. Since $\Phi$ is additive and every object of $\uPerm(\cC; \mu)^{\rm kar}$ is a summand of sums of $\Vec_{\fB(X)}$'s, the result follows.
\end{proof}

\section{The pre-Galois case} \label{s:ww}

\subsection{Set-up} \label{ss:ww-setup}

Let $\cE$ be an order in a pre-Galois category; to be concrete, we take $\cE$ to be an order in the category $\bS(G)$ of finitary smooth $G$-sets, for some pro-oligomorphic group $G$. One easily sees that $\cE$ is a finitely-powered regular category. In \S \ref{s:ww}, we examine this category from the perspective developed in this paper.

The final object $\bone$ of $\cE$ is the one-point set. It has two subobjects, the empty set $E_0$, and itself $E_1$; we thus have $\Lambda=\{0,1\}$. The principal (type~1) objects of $\cE$ are exactly the non-empty $G$-sets in $\cE$. We let $\fG=\fG_1$ be the pro-oligomorphic group associated to $\cE^s_1$, we let $\cC_0=\bS(\fG)$ be the category of finitary smooth $\fG$-sets, and we let $\cC$ be the order in $\cC_0$ consisting of objects whose atoms have the form $\fB(X)$ with $X \in \cE_1$. Proposition~\ref{prop:B-prod} shows that $\cC$ is closed under products, which implies it is an order.

\begin{example}
Suppose $G$ is the trivial group, so that $\cE$ is simply the category of finite sets. In this case, $\fG$ is the group of homeomorphisms of the Cantor set. This case is studied in detail in \cite{cantor}.
\end{example}

\subsection{Degree functions}

We now classify degree functions on $\cE$.

\begin{proposition}
The only degree function on $\cE$ is the trivial one.
\end{proposition}

\begin{proof}
Let $\nu$ be a degree function on $\cE$, and let $f \colon Y \to X$ be a surjection of non-empty objects in $\cE$. Consider the following diagram
\begin{displaymath}
\xymatrix{
Y \ar[r] \ar[d]^f & Y \amalg \bone \ar[d]^{f'} & \bone \ar[l] \ar[d]^{\id} \\
X \ar[r] & X \amalg \bone & \bone \ar[l] }
\end{displaymath}
where $f'=f \amalg \id$. Both squares are cartesian, and so $\nu(f)=\nu(f')=\nu(\id)=1$.
\end{proof}

\begin{remark}
The above proof applies to degree functions valued in monoids, and thus shows that the universal monoid $\Pi(\cE)$ from \S \ref{ss:degree} is trivial.
\end{remark}

\subsection{The first measure}

Let $\alpha$ be the $\bZ$-valued measure on $\cC$ corresponding to the trivial degree function via Construction~\ref{con:main}. We now explicitly compute this measure. For an object $X$ of $\cE$, let $\rho(X)$ denote the number of atoms in $X$.

\begin{proposition}
We have $\alpha(\fB(X))=(-1)^{\rho(X)-1}$ for all non-empty objects $X$.
\end{proposition}

\begin{proof}
By definition, we have
\begin{displaymath}
\alpha(\fB(X))=\sum_{Y \subset X, Y \ne \emptyset} \mu(Y,X).
\end{displaymath}
Let $X=X_1 \sqcup \cdots \sqcup X_n$, where the $X_i$'s are atoms. For a subset $S$ of $[n]$, let $X_S=\bigcup_{i \in S} X_i$. Then the $X_S$ are exactly the subobjects of $X$ in $\cE$, and so the poset $\Sub(X)$ of subobjects is isomorphic to the poset $\Sub([n])$ of subsets of $[n]$. A simply computation then shows that the above sum is equal to $(-1)^{n-1}$.
\end{proof}

\begin{remark} \label{rmk:alpha}
The category $\uPerm(\cC; \alpha)$ always has nilpotents of non-zero trace, and thus does not admit an abelian envelope. To see this, let $\cE'$ be the category of finite sets, regarded as the subcategory of $\cE$ consisting of sets with trivial $G$-action, let $\cC'$ be the corresponding subcategory of $\cC$, and let $\alpha'$ be the restriction of $\alpha$ to $\cC'$. Then $\uPerm(\cC; \alpha)$ contains $\uPerm(\cC'; \alpha')$ as a tensor subcategory. By \cite[\S 6.2]{cantor}, the latter category contains nilpotents of non-zero trace.
\end{remark}

\subsection{The second measure}

We now construct a second measure on $\cC$, in some cases. We begin by characterizing regular measures on $\cE$ valued in $\bF_2$. For this, we introduce the following notion:

\begin{definition}
We say that $\cE$ is \defn{odd} if for every morphism of atoms $X \to Y$ and $X' \to X$, the fiber product $X' \times_X Y$ contains an odd number of atoms.
\end{definition}

\begin{proposition}
The category $\cE$ is odd if and only if it admits a regular measure valued in $\bF_2$; furthermore, such a measure is unique.
\end{proposition}

\begin{proof}
A regular $\bF_2$-valued measure must assign~1 to each object, and is therefore unique. This formula defines a regular measure if and only if $\cE$ is odd: indeed, (a) and (b) of Definition~\ref{defn:meas} clearly hold, while~(c) is equivalent to $\cE$ being odd.
\end{proof}

We can now introduce the second measure.

\begin{proposition}
Suppose that $\cE$ is odd. Then $\cC$ admits a $\bZ$-valued measure $\beta$ satisfying $\beta(\fB(X))=(-2)^{\rho(X)-1}$ for all non-empty $X$ in $\cE$.
\end{proposition}

\begin{proof}
Let $X$ be a non-empty object of $\cE$, and write $X=X_1 \sqcup \cdots \sqcup X_n$ where each $X_i$ is an atom. We define a \defn{type~1 surjection} to be one of the form $X_i \amalg X \to X$, where the two maps are the natural maps; we say $i$ is the \defn{distinguished index}. We define a \defn{type~2 surjection} to be one of the form $\coprod_{i=1}^n f_i$ where $f_i \colon X_i' \to X_i$ is a surjection for each $i$, and $f_i$ is an isomorphism for all but one $i$; we again call this special value of $i$ the \defn{distinguished index}. Every surjection in $\cE$ can be factored into a sequence of type~1 and type~2 surjections.

Let $\fB(X') \to \fB(X)$ and $\fB(Y) \to \fB(X)$ be morphisms of atoms in $\cE$. It suffices to show that
\begin{equation} \label{eq:reg}
\beta(\fB(Y) \times_{\fB(X)} \fB(X')) = \frac{\beta(\fB(X')) \beta(\fB(Y))}{\beta(\fB(X))}.
\end{equation}
By factoring the maps $f \colon X' \to X$ and $g \colon Y \to X$ into type~1 and type~2 surjections, we reduced to the case where $f$ and $g$ are either type~1 or type~2. Write $X=X_1 \sqcup \cdots \sqcup X_n$. If the distinguished indices for $f$ and $g$ are different, then \eqref{eq:reg} is easily verified. We thus assume that $f$ and $g$ both have distinguished index~1. If $f$ and $g$ have different types, then \eqref{eq:reg} is also easily verified. We thus have two cases remaining.

Suppose $f$ and $g$ both have type~1. For notational simplicity, we assume $n=1$; the additional summands in $X$ do not change the argument at all. Thus $f \colon X' \to X$ and $g \colon Y \to X$ are maps of atoms. The right side of \eqref{eq:reg} is~1. Let $Y'=Y \times_X X'$ be the fiber product in $\cE$, and let $Y'_1 \sqcup \cdots \sqcup Y'_m$ be the decomposition of $Y'$ into atoms. For a subset $S$ of $[m]$, let $Y'_S=\bigcup_{i \in S} Y'_i$. Each non-empty subobject of $Y'$ is ample, and so
\begin{displaymath}
\fB(Y) \times_{\fB(X)} \fB(X') = \bigsqcup_{S \subset [m], S \ne \emptyset} \fB(Y'_S).
\end{displaymath}
We thus find
\begin{displaymath}
\beta(\fB(Y) \times_{\fB(X)} \fB(X')) = \sum_{i=1}^m \binom{m}{i} (-2)^{i-1} = 1,
\end{displaymath}
where in the final step we use that $m$ is odd (which is our assumption on $\cE$). Thus \eqref{eq:reg} holds.

We now suppose $f$ and $g$ both have type~2. Again, we assume $n=1$ for notational simplicity. Thus $X$ is an atom and $f$ and $g$ are the maps $X \amalg X \to X$. The fiber product $Y'=Y \times_X X'$ in $\cE$ is $X^{\amalg 4}$, where the $4$ is really $[2] \times [2]$, and the ample subobjects of $Y'$ are the ample subsets of $[2] \times [2]$, of which there are seven. We thus find
\begin{displaymath}
\fB(Y) \times_{\fB(X)} \fB(X') = \fB(X^{\amalg 4}) \amalg \fB(X^{\amalg 3})^{\amalg 4} \amalg \fB(X^{\amalg 2})^{\amalg 2},
\end{displaymath}
and so
\begin{displaymath}
\beta(\fB(Y) \times_{\fB(X)} \fB(X')) = (-2)^3 + 4 (-2)^2 + 2(-2) = 4,
\end{displaymath}
which agrees with the right side of \eqref{eq:reg}.
\end{proof}

\begin{remark}
Suppose $\cE$ is pre-Galois and odd, and let $\mu$ be the unique $\bF_2$-valued regular measure on $\cE$. The category $\cA=\uPerm(\cE; \mu)^{\rm kar}$ is abelian by \cite[Theorem~132]{repst}. Let $k$ be a field of characteristic~0, and regard $\beta$ as $k$-valued. One can show that $\uPerm_k(\cC; \beta)^{\rm kar}$ is the category $k[\cA]^{\sharp}_{\chi}$ appearing in \cite[\S 1.2]{dblexp}, where $(-)^{\sharp}$ denotes the additive--Karoubi envelope and $\chi$ is the trivial character of $\bF_2^{\times}$. In particular, $\uPerm_k(\cC; \beta)^{\rm kar}$ is a semi-simple pre-Tannakian category by \cite[Theorem~1.4]{dblexp}. A detailed proof can be found in \cite[\S 5]{cantor} in the case when $G$ is the trivial group; the general case is not much different.
\end{remark}

\subsection{Regular measures}

Since the measures $\alpha$ and $\beta$ are $\bZ$-valued, they yield measures $\alpha_k$ and $\beta_k$ valued in any ring $k$ by extension of scalars. The measure $\alpha_k$ is regular, while $\beta_k$ is regular provided~2 is invertible in $k$. These account for all regular measures:

\begin{proposition} \label{prop:reg}
Let $\mu$ be a regular $k$-valued measure on $\cC$. Then $\mu$ is either $\alpha_k$ or $\beta_k$.
\end{proposition}

\begin{proof}
For notational simplicity, we put $\mu(X)=\mu(\fB(X))$ and $\mu(n)=\mu(\bone^{\amalg n})$. Let $\cE'$ and $\cC'$ be as in Remark~\ref{rmk:alpha}. The restriction of $\mu$ to $\cC'$ is either $\alpha_k$ or $\beta_k$ by \cite[Theorem~4.5]{cantor}. We show that $\mu(X)=\mu(\rho(X))$ for all non-empty objects $X$, which will complete the proof.

Let $X$ and $Y$ be non-empty objects of $\cE$. Then $X \amalg Y$ is the fiber product of $X \amalg \bone$ and $\bone \amalg Y$ over $\bone \amalg \bone$ in $\cE$, and there are no proper subobjects of $X \amalg Y$ that are ample. We thus have a cartesian square in $\cC$
\begin{displaymath}
\xymatrix{
\fB(X \amalg Y) \ar[r] \ar[d] & \fB(X \amalg \bone) \ar[d] \\
\fB(\bone \amalg Y) \ar[r] & \fB(\bone \amalg \bone) }
\end{displaymath}
and so
\begin{displaymath}
\mu(X \amalg Y) = \frac{\mu(X \amalg \bone) \mu(Y \amalg \bone)}{\mu(2)}.
\end{displaymath}
By iteratively applying this equation, we find
\begin{equation} \label{eq:mu-union}
\mu(X_1 \amalg \cdots \amalg X_n)= \frac{\mu(n)}{\mu(2)^n} \prod_{i=1}^n \mu(X_i \amalg \bone),
\end{equation}
for any non-empty objects $X_1, \ldots, X_n$ of $\cE$, provided $n \ge 2$.

Now, suppose $X$ is an atom of $\cE$. Then $X^{\amalg n}$ is the product of $X$ with $\bone^{\amalg n}$ in $\cE$, and there are no proper subobjects of $X^{\amalg n}$ that are ample (since $X$ is an atom). We thus find
\begin{displaymath}
\mu(X^{\amalg n}) = \mu(X) \mu(n).
\end{displaymath}
By computing the left side using \eqref{eq:mu-union}, we find
\begin{displaymath}
\mu(X) = \bigg( \frac{\mu(X \amalg \bone)}{\mu(2)} \bigg)^n
\end{displaymath}
for all $n \ge 2$. We conclude that $\mu(X)=1=\mu(1)$ and $\mu(X \amalg \bone)=\mu(2)$.

Finally, suppose $X=X_1 \amalg \cdots \amalg X_n$, where each $X_i$ is an atom. If $n=1$ then $X$ is an atom, and we have seen that $\mu(X)=\mu(1)$. If $n \ge 2$ then, applying \eqref{eq:mu-union}, we find $\mu(X)=\mu(n)$. We conclude that $\mu(X)=\mu(\rho(X))$ whenever $X$ is non-empty, which completes the proof.
\end{proof}

\begin{remark}
It is possible to write down a presentation of $\Theta(\cC)$ in terms of $\cE$. However, we do not know of any interesting measures on $\cC$ other than our $\alpha$ and $\beta$.
\end{remark}

\subsection{Fast-growing groups}

The following is the precise version of Theorem~\ref{mainthm3}:

\begin{theorem} \label{thm:fast}
Given any integer sequence $(a_n)_{n \ge 2}$, there exists an oligomorphic group $(\fG, \Omega)$ such that the number of orbits of $\fG$ on $\Omega^n$ exceeds $a_n$ for $n \ge 2$, and $\fG$ admits a regular $\bQ$-valued measure.
\end{theorem}

\begin{proof}
Let $(G, \Xi)$ be an oligomorphic group such that $G$ acts transitively on $\Xi$, and the number of orbits of $G$ on $\Xi^n$ exceeds $a_n$ for $n \ge 2$. Such a group exists by the method of \cite[\S 3.24]{CameronBook}. Let $\cE=\bS(G)$, and let $\cC$ be as in this section. We have $\cC=\bS(\fG; \fU)$ for some pro-olgiomorphic group $\fG$ and stabilizer class $\fU$.

Every atom of $\bS(G)$ is a subquotient of a power of $\Xi$, and so every atom of $\cC$ is a quotient of a subpower of $\fB(\Xi)$. This means that if $\Omega$ is the $\fG$-set corresponding to $\fB(\Xi)$ then $\fG$ acts faithfully on $\Omega$, and so $(\fG, \Omega)$ is an oligomorphic group. One easily sees that the number of orbits of $\fG$ on $\Omega^n$ also exceeds $a_n$. (Note that since $G$ acts transitively on $\Xi$ every non-empty $G$-subset of $\Xi^n$ is ample.)

We have seen that $\cC$ carries a regular $\bQ$-valued measure, namely $\alpha$. By Corollary~\ref{cor:Gclass} the stabilize class $\fU$ ``large,'' meaning every open subgroup contains a member of $\fU$ with finite index. It thus follows that $\alpha$ extends to a regular $\bQ$-valued measure on $\bS(\fG)$ \cite[\S 4.2(d)]{repst}. This completes the proof.
\end{proof}

\subsection{Extending measures}

Let $\cC$ be an order in a pre-Galois category $\cC_0$. Given a regular measure $\mu$ on $\cC$, it is important to know if it can be extended to $\cC_0$; for instance, \cite{discrete} shows that this is a necessary condition for $\uPerm(\cC; \mu)$ to admit an abelian envelope. It is not difficult to construct measures that fail to extend in positive characteristic $p$, essentially due to $p$ dividing the order of automorphism groups in $\cC_1$; in \cite[\S 15.8]{repst}, the measure $\mu_t$ for $t \not\in \bF_p$ is such an exmaple. We now give the first example in characteristic~0 of a regular measure that fails to extend.

\begin{proposition}
There exists an order $\cC$ in a pre-Galois category $\cC_0$ in a pre-Galois category and a regular $\bQ$-valued measure on $\cC$ that does not extend to a measure on $\cC_0$.
\end{proposition}

\begin{proof}
Let $G$ be the homeomorphism group of the Cantor set $\fX$, which is oligomorphic via its action on the set of clopen subsets of $\fX$ \cite[\S 3]{cantor}. Let $\cE_2$ be the category of smooth $G$-sets. For a positive integer $n$, we let $X(n)$ be the set of surjective continuous maps $\fX \to [n]$. These are transitive smooth $G$-sets, and every transitive smooth $G$-set is a quotient of some $X(n)$ by a finite group by \cite[Proposition~3.10]{cantor}, which is just a special case of our Proposition~\ref{prop:Gclass}. We let $\cE_1$ be the order in $\cE_2$ consisting of $G$-sets whose orbits are of the form $X(n)$.

We let $\cC_1$ and $\cC_2$ be the categories obtained from $\cE_1$ and $\cE_2$ as in \S \ref{ss:ww-setup}. Since every object of $\cE_2$ is a quotient of an object from $\cE_1$, it follows that every atom of $\cC_2$ admits a (necessarily surjective) map from an atom in $\cC_1$, and so $\cC_1$ is an order in $\cC_2$; let $\cC_0$ be the common pre-Galois category. The category $\cE_1$ is odd; indeed it admits a regular $\bZ$-valued measure by $X(n) \mapsto (-1)^{n-1}$ \cite[Theorem~4.2]{cantor}, which reduces to a regular $\bF_2$-valued measure. The category $\cE_2$ does not admit any $\bF_2$-valued measure by \cite[Theorem~4.5]{cantor}, and is thus not odd. Thus the $\beta$ measure for $\cC=\cC_1$ does not extend to $\cC_2$; indeed, any extension would be regular, but $\cC_1$ only carries the regular measure $\alpha$ by Proposition~\ref{prop:reg}, which does not extend $\beta$. It thus does not extend to $\cC_0$ either.
\end{proof}

\end{document}